\documentclass[final]{elsarticle}
\usepackage{amsmath, amssymb, amsthm}
\usepackage{placeins}
\usepackage[ruled,vlined,longend,titlenotnumbered]{algorithm2e}
\usepackage{algpseudocode}
\usepackage{comment}
\usepackage{mathtools}
\usepackage{nccmath}
\usepackage{tikz}
\newcommand*\circled[1]{\tikz[baseline=(char.base)]{
            \node[shape=circle,draw,inner sep=1pt] (char) {#1};}}

\usepackage{lineno}

\usepackage{subcaption}

\usepackage[hypcap=false]{caption}
\usepackage[list=true]{subcaption}

\makeatletter
\let\c@table\c@figure % for (1)
\let\ftype@table\ftype@figure % for (2)
\makeatother

\usepackage[hypertexnames=false,colorlinks=true,breaklinks=true,bookmarks=true,urlcolor=blue,citecolor=blue,linkcolor=blue,bookmarksopen=false,draft=false]{hyperref}

\modulolinenumbers[5]

\DeclareMathOperator{\Diag}{Diag}

\DeclareMathOperator{\ldet}{ldet}

\DeclareMathOperator*{\argmax}{argmax}

\makeatletter
\renewcommand*{\top}{%
  {\mathpalette\@transpose{}}%
}
\newcommand*{\@transpose}[2]{%
  \raisebox{\depth}{$\m@th#1\mathsf{T}$}%
}
\makeatother

\newtheorem{theorem}{Theorem}
\newtheorem{example}{Example}
\newtheorem{lemma}[theorem]{Lemma}
\newtheorem{proposition}[theorem]{Proposition}
\newtheorem{definition}{Definition}

\newtheorem*{remark}{Remark}

\def\smallColSep{\setlength{\arraycolsep}{1pt}}

\begin{document}
\begin{frontmatter}
  \title{Tridiagonal Maximum-Entropy Sampling  \\ and Tridiagonal Masks$\vphantom{|}^1$\fnref{myfootnote0}}
  \fntext[myfootnote0]{A short preliminary version of 
  part of this work appeared in the proceedings of \emph{LAGOS 2021}
  (XI Latin and American Algorithms, Graphs and Optimization Symposium).}
  
  \author{Hessa Al-Thani\fnref{myfootnote1}}
\fntext[myfootnote1]{The work of H. Al-Thani was made possible by a Graduate Sponsorship Research Award from the    Qatar National Research Fund (a member of Qatar Foundation). The findings herein reflect the work,     and are solely the responsibility, of the authors.}

  \author{Jon Lee\fnref{myfootnote2}}
\address{IOE Department, University of Michigan,
  Ann Arbor, Michigan, USA}
\fntext[myfootnote2]{J. Lee was supported in part by AFOSR grants FA9550-19-1-0175 and FA9550-22-1-0172.}

%   \author{Hessa Al-Thani\thanksref{H}\thanksref{hemail}}
%   \address{IOE Department\\ University of Michigan\\
%   Ann Arbor, Michigan, USA} \author{Jon Lee\thanksref{JL}\thanksref{jemail}}
%  \address{IOE Department\\ University of Michigan\\
%   Ann Arbor, Michigan, USA} 
% \thanks[H]{The work of H. Al-Thani was made possible by a Graduate Sponsorship Research Award from the    Qatar National Research Fund (a member of Qatar Foundation). The findings herein reflect the work,     and are solely the responsibility, of the authors.}
% \thanks[JL]{J. Lee was supported in part by AFOSR grant FA9550-19-1-0175.}
% \thanks[hemail]{Email:
%     \href{mailto:hessakh@umich.edu} {\texttt{\normalshape
%       hessakh@umich.edu}}} 
% \thanks[jemail]{Email:
%      \href{mailto:jonxlee@umich.edu} {\texttt{\normalshape
%         jonxlee@umich.edu}}} 

\begin{abstract} 
The NP-hard maximum-entropy sampling problem (MESP) seeks a maximum (log-)determinant 
principal submatrix, of a given order, from an
input covariance matrix $C$. 

We give an efficient dynamic-programming algorithm for MESP
when $C$ (or its inverse) is tridiagonal and generalize it to
the situation where the support graph of $C$ (or its inverse)
is a spider graph with a constant number of legs (and beyond). We give a class of 
arrowhead covariance matrices $C$ for which a natural greedy algorithm solves MESP.

A \emph{mask} $M$  for MESP is a correlation
matrix with which  we pre-process $C$, by taking the Hadamard product  $M\circ C$.
Upper bounds on MESP with $M\circ C$ give upper bounds on MESP with $C$.
Most upper-bounding methods are much faster to apply, when the input matrix is tridiagonal,
so we consider tridiagonal masks $M$ (which yield tridiagonal $M\circ C$). We make a detailed analysis of such tridiagonal masks, and 
develop a combinatorial local-search based upper-bounding method that takes advantage  of 
fast computations on tridiagonal matrices. 
\end{abstract}

\journal{Discrete Applied Mathematics}
\date{revised \today}

\end{frontmatter}

%\linenumbers

\section{Introduction}\label{intro}

Let $C$ be an order-$n$ symmetric positive semidefinite real matrix, and let $s$ be an integer
satisfying $1\leq s\leq n$. In applications, $C$ is the covariance matrix
for a multivariate Gaussian random vector $Y_{N_n}$, where  $N_n:=[1,n]:=\{1,2,\ldots,n\}$
--- note that we regard $N_n$ as an ordered set, which will be important later. 
For nonempty $S\subseteq N_n$, let $C[S,S]$ denote the principal submatrix of $C$ indexed by $S$.
We denote $\log \det(\cdot)$ by $\ldet(\cdot)$. Up to constants, $\ldet C[S,S]$
is the (differential) entropy associated with the subvector $Y_S$ (see \cite{SW}, for example).
The \emph{maximum-entropy sampling problem}, defined by \cite{SW}, 
is
\begin{equation}
z=\max~ \left\{  \ldet C[S,S] ~:~ |S|=s,~ S\subseteq N_n\right\}\tag{MESP}\label{MESP}
\end{equation}
\ref{MESP} corresponds to choosing a maximum-entropy subvector $Y_S$ from $Y_N$,
subject to $|S|=s$. Later, it will sometimes be convenient to emphasize the dependence
of \ref{MESP} on the data, and in such situations we will write \ref{MESP}$(C,s)$ and $z(C,s)$.

\ref{MESP} was introduced in \cite{SW}, 
it was established to be NP-Hard in \cite{KLQ} (via reduction from the NP-Complete problem: does an $n$-vertex graph $G$ have a vertex packing of cardinality $s$), 
and there has been considerable work on algorithms. 
Viable approaches aimed at
exact solution of moderate-sized instances
employ branch-and-bound (see \cite{KLQ}). For use in
branch-and-bound, we have many methods for efficiently calculating good upper bounds; see
\cite{KLQ,AFLW_IPCO,AFLW_Using,HLW,LeeWilliamsILP,AnstreicherLee_Masked,BurerLee,Anstreicher_BQP_entropy,Kurt_linx,Mixing,Nikolov,Weijun,fact}
and the survey \cite{LeeEnv}. Notably, we have developed an \texttt{R} package
providing an easy means for instantiating \ref{MESP} from raw environmental-monitoring data
(see \cite{Rpackage,Al-ThaniLee1}).

In what follows, ``$\circ$'' denotes Hadamard (i.e., elementwise) product
of a pair of matrices of the same dimensions.
Given a positive integer $n$, we define a \emph{mask} as any $n\times n$ symmetric positive semidefinite matrix
$M$ with all ones on the diagonal. Masks are better known as ``correlation matrices''. 
These were introduced for \ref{MESP} in \cite{AnstreicherLee_Masked}, where it was observed that for any $S$ and mask $M$,
$\ldet C[S,S] \leq \ldet (C\circ M)[S,S]$, and so any upper bound on \ref{MESP} for
$C\circ M$ is valid for the \ref{MESP} on $C$.
It is a challenge to find a good mask with respect to a
particular \ref{MESP} upper-bounding method --- i.e., a mask $M$ that minimizes the upper bound on $z(C\circ M,s)$. This topic was investigated in \cite{AnstreicherLee_Masked}
and \cite{BurerLee}.
We define a \emph{combinatorial mask}
as any block-diagonal mask $M$ where the diagonal blocks (which may vary in size) are
matrices of all ones. At the extremes, we have $M:=I_n$ (an identity matrix) and $M=:J_n$ (an all ones matrix). As we observed in 
\cite{AnstreicherLee_Masked}, we can view some of \cite{HLW} and  \cite{LeeWilliamsILP} in this way; 
e.g., the ``spectral partition bound'' of \cite{HLW}. But there are many possibilities for masks that 
are not combinatorial masks. For example, a (tridiagonal) \emph{$\frac{1}{2}$-mask} $M$ has $M_{i,i+1}=M_{i+1,i}:= 
\frac{1}{2}$ for $i=1,\ldots,n-1$ (and all other off-diagonal entries are 0). 
One goal of ours is 
to investigate computational aspects of tridiagonal masks $M$, so as to take advantage
of the fact that most bounds can be calculated much faster (than on dense matrices)
for tridiagonal matrices (in our case,
$M\circ C$). 

%We also note that if a mask $M$ is permutation invariant (like $I_n$ or $J_n$), 
%then for a permutation invariant mask, there can be no gain in the bound by permuting the mask. 

In Section \ref{sec:tridiagC}, we give an $\mathcal{O}(n^5)$ dynamic-programming algorithm for \ref{MESP} when $C$ is 
tridiagonal. This also solves the problem in the same complexity when instead $C^{-1}$
is tridiagonal. This is the first progress on identifying significant polynomially-solvable
cases of \ref{MESP}. We extend our result to the case where the
support graph of $C$ is a spider with a constant number of legs, and we indicate how
it can be further extended when the number of connected components
of the support graph of $C$ is polynomial in $n$. 
Finally, we present the results of computational experiments
on matrices having support graphs that are spiders, 
indicating the superiority of a parallel implementation of
our dynamic-programming algorithm as compared with branch-and-bound.

In Section \ref{sec:stars}, we characterize a class of 
positive-semidefinite ``arrowhead matrices''
(i.e., having the support graph of $C$ being a star), such that a certain natural
greedy algorithm optimally solves \ref{MESP}. 

In Section \ref{sec:tridiagM},
we characterize for each $n$, masks that differ from $\frac{1}{2}$-masks in two (symmetric) pairs
of off-diagonal positions. These results are useful in identifying good masks for \ref{MESP} bounds, in the context of local search in the space of masks.

In  Section \ref{sec:masklocalsearch}, we develop a combinatorial local-search algorithm that seeks a good
tridiagonal mask $M$ for the so-called `linx' upper bound (which is one of the best known upper-bounding method for \ref{MESP}).
Some computational experiments validate our approach.

\section{Tridiagonal covariance matrices}\label{sec:tridiagC} 

The inverse covariance matrix, known as the \emph{precision matrix}, 
captures the conditional covariances between pairs of random variables,
conditioning on the remaining $n-2$ random variables. 
For a spatial process with random variables  placed in a line, a reasonable 
model may have the precision matrix being tridiagonal, 
when the random variables are numbered in a 
natural order (along the line);
intuitively, such a model would assume that  no extra information can be obtained from a non-neighbor of a random variable $Y_i$, over and above the 
information obtainable from the neighbors of $Y_i$. 

To see how the inverse covariance matrix
plays a role for \ref{MESP}, we employ the   identity 
\[
\det C[S,S] = \det C \times \det C^{-1}[N_n\setminus S,N_n\setminus S] 
\]
(see \cite[Section 0.8.4]{HJBook}). 
With this identity,  we have $z(C,s)=\ldet C + z(C^{-1},n-s)$, and so
the \ref{MESP} for choosing $s$ elements with respect to $C$ is
equivalent to the \ref{MESP}  for choosing $n-s$ elements with respect to $C^{-1}$.

The determinant of a tridiagonal matrix %$T_r\in\mathbb{R}^{r\times r}$
can be calculated in linear time, via a simple recursion  (see \cite{demmel1997}),
which we write for the symmetric case as that is our need.
Let $T_1=(a_1)$, and  for $r\geq 2$, let
\[
T_r:=\left(
  \begin{array}{ccccc}
    a_1 & b_1 &  &  &  \\
    b_1 & a_2 & b_2 &  &  \\
     & b_2 & \ddots & \ddots &  \\
     &  & \ddots & \ddots & b_{r-1} \\
     &  &  & b_{r-1} & a_r \\
  \end{array}
\right).
\]
\begin{lemma}\label{lem:tridet}
Defining  $\det T_0:=1$, we have 
 $\det T_r =: a_r \det T_{r-1} - b_{r-1}^2 \det T_{r-2}$,
for  $r\geq 2$. 
\end{lemma}

% \{\color{red} T is $r \times r$ but $a$ and $b$ are in terms of $n$, I think T should be $n \times n$ I've changed it, let me know if you disagree. Also should we cite this formula?}

\begin{theorem}
\ref{MESP} is polynomially solvable when $C$ or $C^{-1}$ is tridiagonal, or when there is a symmetric
permutation of $C$ or $C^{-1}$ so that it is tridiagonal. 
\end{theorem}

\begin{proof} 
Suppose that $C$ is tridiagonal. Let
$S$ be an ordered subset of $N_n$. Then we can write 
$C[S,S]$ uniquely as  
$C[S,S]=\Diag(C[S_1,S_1],C[S_2,S_2], \ldots, C[S_p,S_p])$,
with $p\geq 1$,
where each $S_i$ is a \emph{maximal ordered contiguous subset} of $S$, 
and for all $1\leq i<j\leq p$, all elements of $S_i$  are less than all elements of $S_j$. 
We call the $S_i$ the \emph{pieces} of $S$, and in particular $S_p$ is the \emph{last piece}. 
It is easy to see
that 
\[
\det C[S,S] = \textstyle  \prod_{i=1}^p \det C[S_i,S_i] = \det C[S_p,S_p] \times \det C[S\setminus S_p, S\setminus S_p]~.
\]
Of course every $S$ has a last piece, and for an optimal $S$ to \ref{MESP}, if the last piece is $S_p=:[k,\ell]$,
then we have the ``principle of optimality'':
\[
\ldet C[S\setminus [k,\ell],S\setminus [k,\ell]]= z(C[N_{k-2},N_{k-2}], s-(\ell-k+1)).
\]

With this way of thinking, we define
\begin{align*} 
f(k,\ell,t):=& \max \Big\{
\ldet C[S,S] ~:~ \\
&\qquad\qquad |S|=t,~ S\subset N_n~, \mbox{ and the last piece of $S$ is } [k,\ell]
\Big\},
\end{align*}
for $1\leq \ell-k+1\leq t \leq s$.
Note that the last block of $S$ (in the definition of $f$) has $\ell-k+1$ elements,
which is why $t$ must be at least that large. 
We have that
\[
z(C,s) = \max_{k,\ell} \left\{ 
f(k,\ell,s) ~:~ 1\leq k\leq \ell\leq n,~ \ell-k+1\leq s 
\right\},
\]
where we are simply maximizing over the possible (quadratic number of) last pieces. 

Our dynamic-programming recursion is
\begin{align*}
f(k,\ell,t) = \ldet C[[k,\ell],[k,\ell]] + &
\max_{i, j} \left\{ 
f(i,j,t-(\ell-k+1)) ~:~ \right. \\
& \left. 1\leq i\ \leq j\leq k-2,~ j-i+1 \leq t-(\ell-k+1)
\right\},
\end{align*}
The idea is that if $[k,\ell]$ is the last block of an 
optimal selection of $t$ elements, then we have to pick $t-(\ell-k+1)$ more elements,
and \emph{element $k-1$ is ruled out}  (because $[k,\ell]$ is maximal).

To get the recursion started, we calculate
\[
f(k,\ell,\ell-k+1)=\ldet C[[k,\ell],[k,\ell]],
\]
for $1\leq k \leq \ell \leq n$, $\ell-k+1\leq s$, 
observing that in such a boundary case, there is only one feasible solution.
Already, it appears that the initialization requires $\mathcal{O}(n^5)$ basic arithmetic operations,
but in fact we can do this part in $\mathcal{O}(n^2)$ operations, using the tridiagonal-determinant formula (Lemma \ref{lem:tridet}).

Next, we compute, using the recursion, for $t=1,2,\ldots,s$,
$f(k,\ell,t)$ for all $1\leq k \leq \ell \leq n$ such that $\ell-k+1 <t$.
It is not hard to see that this gives an $\mathcal{O}(n^5)$ algorithm for \ref{MESP}, when 
$C$ is tridiagonal, and by earlier observations, we also get an 
$\mathcal{O}(n^5)$ algorithm for \ref{MESP}, when  $C^{-1}$ is tridiagonal. 
\end{proof}

For an arbitrary symmetric $C$, with rows and columns indexed from $N_n$\thinspace,
we consider the \emph{support graph} $G(C)$, with node set $V(G[C]):=N_n$\thinspace, and
edge set 
$E(G(C)):=\{(i,j) ~:~ i,j\in N_n\thinspace,~ i<j,~ C[i,j]\not= 0\}$. 
If $C$ is a generic tridiagonal matrix (i.e., $C[i,i+1]\not=0$ for $i\in N_{n-1}$), then  $G(C)$ is the
path $P_n:=$ \circled{1}--\circled{2}-- $\cdots$--\circled{$n$}.  
If $C$ is tridiagonal but
not generic, then $G(C)$ is a subgraph of $P_n$\thinspace. 
Because we can efficiently solve \ref{MESP} when $C$ is tridiagonal (and also when $C$ is tridiagonal after symmetric permutation),
it is natural to consider broader classes of $C$, and natural exploitable structure is
encoded in the support graph $G(C)$. 

We are interested in ``spiders'' with $r\geq 1$ legs
on an n-vertex set (see \cite{MARTIN2008}, for example): for convenience,
we let the vertex set be $N_n$, and let vertex 1 be the \emph{body} of the spider; the non-body vertex set $V_i$ 
of \emph{leg} $i$, is a non-empty 
contiguously numbered subset of $N_n \setminus\{1\}$,
such that distinct $V_i$ do not intersect, and the union of all $V_i$ is $N_n\setminus \{1\}$; 
 we number the legs $i$ in such a way that: (i) the minimum element of $V_1$ is 2, and (ii)
the minimum element of $V_{i+1}$ is 
one plus the maximum element of $V_i$, for $i\in[1,r-1]$. A small example clarifies all of this; see Figures \ref{fig:spider}--\ref{fig:matrix}.
Note that with one leg, the spider is the path $P_n$\thinspace, and with two 
legs, the vertices can be re-numbered so that the the spider is the path $P_n$\thinspace. So, we may as well assume that the spider has $r\geq 3$ legs.

\begin{figure}[ht]
  \centering
\includegraphics[width =0.7\textwidth]{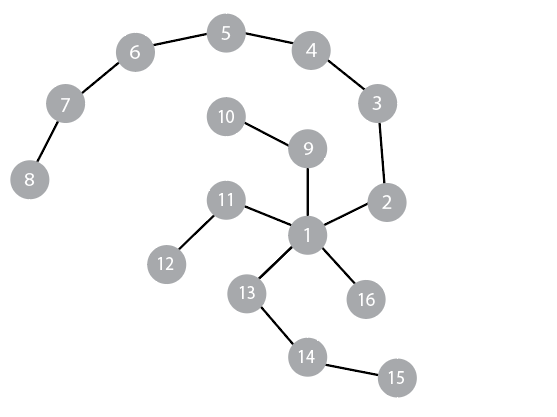}
  \caption{Spider with five legs}\label{fig:spider}
\end{figure}

\begin{figure}
\[\begingroup\smallColSep
\left[\begin{array}{cccccccccccccccc}
\scriptscriptstyle{(1,1)}  &  \scriptscriptstyle{(1,2)}  & \scriptstyle{0} & \scriptstyle{0}& \scriptstyle{0} & \scriptstyle{0}& \scriptstyle{0}& \scriptstyle{0}  &\scriptscriptstyle{(1,9)} & \scriptstyle{0}    & \scriptscriptstyle{(1,11)} &\scriptstyle{0}& \scriptscriptstyle{(1,13)}&\scriptstyle{0}&\scriptstyle{0}&\scriptscriptstyle{(1,16)}\\ %1
     \scriptscriptstyle{(2,1)} & \scriptscriptstyle{(2,2)}& \scriptscriptstyle{(2,3)} &  \scriptstyle{0} & \scriptstyle{0}& \scriptstyle{0}&\scriptstyle{0} & \scriptstyle{0} & \scriptstyle{0} & \scriptstyle{0} & \scriptstyle{0} &\scriptstyle{0}&\scriptstyle{0}&\scriptstyle{0}&\scriptstyle{0}&\scriptstyle{0}\\         %2
     \scriptstyle{0} & \scriptscriptstyle{(3,2)} & \scriptscriptstyle{(3,3)} & \scriptscriptstyle{(3,4)} & \scriptstyle{0}& \scriptstyle{0}&\scriptstyle{0} & \scriptstyle{0}&\scriptstyle{0} & \scriptstyle{0} & \scriptstyle{0}      &\scriptstyle{0}&\scriptstyle{0}&\scriptstyle{0}&\scriptstyle{0}&\scriptstyle{0}\\    %3
      \scriptstyle{0}  & \scriptstyle{0}& \scriptscriptstyle{(4,3)} &      \scriptscriptstyle{(4,4)} & \scriptscriptstyle{(4,5)} & \scriptstyle{0}& \scriptstyle{0}  &\scriptstyle{0} &\scriptstyle{0} & \scriptstyle{0} & \scriptstyle{0}            &\scriptstyle{0}&\scriptstyle{0}&\scriptstyle{0}&\scriptstyle{0}&\scriptstyle{0}\\ %4
      
    \scriptstyle{0}    & \scriptstyle{0}& \scriptstyle{0}& \scriptscriptstyle{(5,4)} &       \scriptscriptstyle{(5,5)}  & \scriptscriptstyle{(5,6)}  &  \scriptstyle{0} & \scriptstyle{0} & \scriptstyle{0}&\scriptstyle{0} & \scriptstyle{0}    &\scriptstyle{0}&\scriptstyle{0}&\scriptstyle{0}&\scriptstyle{0}&\scriptstyle{0}\\ %5
    \scriptstyle{0}       & \scriptstyle{0} & \scriptstyle{0} & \scriptstyle{0}& \scriptscriptstyle{(6,5)}  &  \scriptscriptstyle{(6,6)} & \scriptscriptstyle{(6,7)}  & \scriptstyle{0} &\scriptstyle{0} & \scriptstyle{0} & \scriptstyle{0}& \scriptstyle{0} & \scriptstyle{0}& \scriptstyle{0}& \scriptstyle{0}&\scriptstyle{0} \\ %6
    \scriptstyle{0}       &\scriptstyle{0} & \scriptstyle{0}& \scriptstyle{0}& \scriptstyle{0}& \scriptscriptstyle{(7,8)}  &    \scriptscriptstyle{(7,7)}  & \scriptscriptstyle{(7,8)}  & \scriptstyle{0}& \scriptstyle{0}& \scriptstyle{0}&\scriptstyle{0}      &\scriptstyle{0}&\scriptstyle{0}&\scriptstyle{0}&\scriptstyle{0}\\ %7
    \scriptstyle{0}  & \scriptstyle{0} & \scriptstyle{0} & \scriptstyle{0} & \scriptstyle{0} & \scriptstyle{0} & \scriptscriptstyle{(8,7)}   & \scriptscriptstyle{(8,8)}  & \scriptstyle{0}  & \scriptstyle{0} & \scriptstyle{0}& \scriptstyle{0}& \scriptstyle{0}&\scriptstyle{0}&\scriptstyle{0}&\scriptstyle{0}\\ %8
    \scriptscriptstyle{(9,1)} & \scriptstyle{0}& \scriptstyle{0} & \scriptstyle{0} & \scriptstyle{0} & \scriptstyle{0} & \scriptstyle{0} & \scriptstyle{0}   & \scriptscriptstyle{(9,9)}  & \scriptscriptstyle{(9,10)}  & \scriptstyle{0}& \scriptstyle{0} & \scriptstyle{0} &\scriptstyle{0}&\scriptstyle{0}&\scriptstyle{0}\\ %9
      \scriptstyle{0}     & \scriptstyle{0} & \scriptstyle{0} &\scriptstyle{0} & \scriptstyle{0}& \scriptstyle{0}&\scriptstyle{0} & \scriptstyle{0} & \scriptscriptstyle{(10,9)}  &\scriptscriptstyle{(10,10)}  & \scriptstyle{0}   & \scriptstyle{0}    &\scriptstyle{0}&\scriptstyle{0}&\scriptstyle{0}&\scriptstyle{0}\\  %10
     \scriptscriptstyle{(11,1)}  & \scriptstyle{0}& \scriptstyle{0}&  \scriptstyle{0}&\scriptstyle{0} & \scriptstyle{0}& \scriptstyle{0} & \scriptstyle{0} & \scriptstyle{0} &    \scriptstyle{0}   &   \scriptscriptstyle{(11,11)}    & \scriptscriptstyle{(11,12)}  &\scriptstyle{0}&\scriptstyle{0}&\scriptstyle{0}&\scriptstyle{0}\\ %11
     \scriptstyle{0}   & \scriptstyle{0}& \scriptstyle{0}& \scriptstyle{0}&  \scriptstyle{0} & \scriptstyle{0} & \scriptstyle{0}  & \scriptstyle{0} & \scriptstyle{0} & \scriptstyle{0}& \scriptscriptstyle{(12,11)}  &   \scriptscriptstyle{(12,12)} &  \scriptstyle{0} & \scriptstyle{0}& \scriptstyle{0}&\scriptstyle{0} \\   %12
   \scriptscriptstyle{(13,1)}       & \scriptstyle{0} & \scriptstyle{0} & \scriptstyle{0}&  \scriptstyle{0} & \scriptstyle{0} & \scriptstyle{0} & \scriptstyle{0} & \scriptstyle{0} & \scriptstyle{0}& \scriptstyle{0} & \scriptstyle{0} & \scriptscriptstyle{(13,13)} &        \scriptscriptstyle{(13,14)} &   \scriptstyle{0}& \scriptstyle{0} \\  %13
          \scriptstyle{0} &\scriptstyle{0} & \scriptstyle{0}& \scriptstyle{0}& \scriptstyle{0}&  \scriptstyle{0} & \scriptstyle{0} & \scriptstyle{0} & \scriptstyle{0} & \scriptstyle{0}& \scriptstyle{0} & \scriptstyle{0}& \scriptscriptstyle{(14,13)}  & \scriptscriptstyle{(14,14)} & \scriptscriptstyle{(14,15)}     &\scriptstyle{0}   \\  %14
        \scriptstyle{0} &  \scriptstyle{0} & \scriptstyle{0}&  \scriptstyle{0}& \scriptstyle{0} & \scriptstyle{0} & \scriptstyle{0} & \scriptstyle{0} & \scriptstyle{0} & \scriptstyle{0}& \scriptstyle{0} & \scriptstyle{0}&    \scriptstyle{0}   & \scriptscriptstyle{(15,14)}  & \scriptscriptstyle{(15,15)}   & \scriptstyle{0}   \\  %15
      \scriptscriptstyle{(16,1)}   & \scriptstyle{0}& \scriptstyle{0}& \scriptstyle{0}& \scriptstyle{0} & \scriptstyle{0} & \scriptstyle{0} & \scriptstyle{0} & \scriptstyle{0} & \scriptstyle{0}& \scriptstyle{0} & \scriptstyle{0}& \scriptstyle{0}& \scriptstyle{0}& \scriptstyle{0}           &\scriptscriptstyle{(16,16)}      %16
\end{array} \right]
\endgroup\]
  \caption{Matrix structure corresponding to Figure \ref{fig:spider}}\label{fig:matrix}
\end{figure}

Consider how a \ref{MESP} solution $S$ intersects with the vertices of the spider.
As before, the solution has pieces. Note how at most one piece
contains  the body, and every other piece is a contiguous set
of vertices of a leg. The number of distinct possible pieces 
containing the body is $\mathcal{O}(n^r)$. And the number
of other pieces is $\mathcal{O}(n^2)$. Overall, we have 
$\mathcal{O}(n^r)$ pieces. In any solution, we can order the pieces
by the minimum vertex in each piece. Based on this, we have a 
well defined last piece. From this, we can devise an efficient dynamic-programming algorithm, when we consider $r$ to be constant. 

\begin{theorem}
\ref{MESP} is polynomially solvable when $G(C)$ or $G(C^{-1})$ is a  spider with a constant number of legs.
\end{theorem}

In fact, we can easily organize a dynamic-programming scheme to exploit parallel computation. We can compute $\ldet C[S_1,S_1]$ for all possible pieces containing the body 1. In parallel, we can compute optimal \ref{MESP} solution values for all possible budgets $t\leq s$ for each leg, keeping track of the minimum vertex used for each such solution. With all of that information, we can then calculate an overall optimal solution. 

We conducted some experiments to get some evidence that our dynamic-programming
algorithm can be practical, compared to branch-and-bound.
For each of $k=13,18,23,28,33,38,43$, we constructed ten  positive-semidefinite 
matrices $C$, with $G(C)$ being a spider with three legs and $k$ vertices per leg. So we have $n=3k+1$, and we chose $s\sim \ell n/4$, for $\ell=1,2,3$.
% \footnote{From the na\"{\i}ve combinatorial point of view,
% $s=\lfloor n/2 \rceil$ maximizes $\binom{n}{s}$,
% so these are the hardest by brute-force enumeration. 
% Additionally,  there is empirical evidence that values of 
% $s$ near the middle of the range are substantially harder
% for some-sophisticated exact algorithms than $s$ close to 0 or $n$. See, for example
% \cite[Tables 1--3 ]{Kurt_linx}. \label{note1}}
For the matrix constructions, we built a \texttt{cvx}  (see \cite{cvx}) semidefinite-programming model 
that takes a random choice of positive diagonal elements (constructed using \texttt{rand} of \texttt{Matlab}). The objective of the semidefinite-program was to maximize the sum of the off-diagonal elements of $C$, subject to  $C$ is positive semidefinite, and $G(C)$ being the spider described above. 

In Table \ref{tab:spidercomp}, we report on our numerical experiments. 
We compare our parallel dynamic-programming algorithm,
implemented in \texttt{Matlab}, 
with the serial branch-and-bound  \texttt{Matlab} code of \cite{Kurt_linx}, 
which in turn uses the conic solver \texttt{SDPT3} (see \cite{Toh2012,SDPT3}). 
For our dynamic-programming algorithm, we use the \texttt{Matlab} parallel for-loop instruction \texttt{parfor}, for easy parallelization across spider legs.
Parallelizing branch-and-bound would be a much more difficult task, and
load balancing is highly nontrivial.
We note that the running time of \texttt{Matlab} is highly variable (particularly for  branch-and-bound), even with the same data and deterministic implementations of the algorithms.
But because we average over 
 ten experiments for each choice of $n$ and $s$, our results are meaningful. 
In the individual experiments it is clear that the dynamic-programming algorithm
has more consistent time performance, while the time taken by the branch-and-bound 
fluctuates considerably.
For most of the experiments, we set an upper limit of 2 hours (7200 seconds),
except for the hardest one, the $n=130$, $s=97$ experiment, where we increased the time limit to 4 hours. In the table, $*$ indicates that the time limit was reached. 
Overall the dynamic-programming algorithm performed much better than branch-and-bound,
% , with some exceptions, 
and it scales much better.
% Already for the case  of $n=100$,
% branch-and-bound took more than 2.5 hours for eight cases.

%Setting $s = \floor{\frac{n}{2}}$ 
%%%%%%%%%%%%%%%%%%%%%%%%%%%%%%%%%%%%%%%%%%%%%%%%
%spider table
%still running n = 55 

\begin{table}[ht!]
 \centering
\begin{tabular}{ |c|c|c|c|c|c| } 
 \hline
   &  & No. of  & Avg. time for  & Avg. time for \\ 
   n (k)& s & trials &  DP (sec.) &  B\&B   (sec.) \\ 
 \hline
 40 (13) & 10 & 10  & 
 %0.70  
 $<1$
 & 90  \\ 
%  \hline
 40 (13) & 20 & 10  & 
 %4.5 
 5
 & 21  \\ 
 % \hline
 40 (13) & 30 & 10  & 7 & 678  \\ 
 \hline
 55 (18) & 13 & 10  & 
 %1.26 
 1
 & 5642 \\ 
 55 (18) & 27 & 10  & 21 & 173 \\ 
 55 (18) & 41 & 10  & 33 & 8150 \\ 
 \hline
  70 (23)& 17 & 10  & 
  %3.73 
  4
  & $7200^*$  \\ 
 70 (23)& 35 & 10  & 86 & 241  \\ 
  70 (23)& 52 & 10  & 146 & $7200^*$  \\ 
  \hline
 85 (28)& 21 & 10 & 12 &  $7200^*$  \\ 
  85 (28)& 42 & 10 & 260 &  347  \\ 
   85 (28)& 63 & 10 & 519 &  $7200^*$  \\ 
  \hline
 100 (33)& 25 & 10 & 32 & $7200^*$  \\ 
 100 (33)& 50 & 10 & 858 &  $7200^*$ \\ 
 100 (33)& 75 & 10 & 1364 & $7200^*$  \\ 
 \hline
 115 (38) & 28 & 10 & 68 & $7200^*$ \\ 
 115 (38) & 57 & 10 & 1654 & $7200^*$ \\ 
 115 (38) & 86 & 10 & 3250 & $7200^*$ \\ 
 \hline
  130 (43) & 32 & 10 & 163 & $7200^*$ \\ 
  130 (43) & 65 & 10 & 4086 & $7200^*$ \\ 
  130 (43) & 97 & 10 & 7093 & $14400^*$ \\ 
 \hline
\end{tabular}
\caption{DP vs B{\&}B on spiders}
\label{tab:spidercomp}
\end{table}

\FloatBarrier

% \begin{table}[h!]
% \centering
% \begin{tabular}{c c c c} 
%  \hline
%  Col1 & Col2 & Col2 & Col3 \\ [0.5ex] 
%  \hline
%  1 & 6 & 87837 & 787 \\ 
%  2 & 7 & 78 & 5415 \\
%  3 & 545 & 778 & 7507 \\
%  4 & 545 & 18744 & 7560 \\
%  5 & 88 & 788 & 6344 \\ [1ex] 

% \end{tabular}

% \end{table}

% Suppose that $S$ is a subset of $N$ with $|S|=s$. As above, we have that $S$
% decomposed into pieces, and at most one piece can use the body. 
% The number of different subsets of $S$ that can correspond to 
% a piece using the body is $\mathcal{O}(n^r)$. Assuming that $r$ is 
% constants, we can enumerate these pieces in polynomial time. 
% Deleting any such piece, the remaining problem is now
% a tridiagonal problem, which we can solve in $\mathcal{O}(n^5)$ time.
% So we easily get an overall running time of  $\mathcal{O}(n^{r+5})$.

Finally, we observe that for any class of $n$-vertex graphs for which the number of \emph{connected components} is bounded above by a polynomial in $n$, we can use the same ideas as above to build an efficient dynamic-programming algorithm, assuming that we can enumerate the connected components in polynomial time. We do note that the class of $n$-vertex ``stars'' (i.e., spiders with $n-1$ single-edge legs)  has $2^{n-1}$ subtrees
containing the body,
and so we do not in this way get an efficient algorithm for \ref{MESP} when $G(C)$ is a star. But see the next section for an efficient algorithm for a subclass of the positive-semidefinite $C$ for which 
$G(C)$ is a star. 

We could perhaps hope that what is generally needed for a tree (to lead to an efficient dynamic programming algorithm of this type)
is a degree bound. But even for a degree bound of three, we get bad behavior: for binary trees, the number of subtrees can be exponential in the number of vertices.\footnote{Consider a full binary tree with $\ell$ levels, and hence
$2^\ell-1$ vertices. Such a graph has $2^{\ell-1}$ chains (on $\ell-1$ edges)
from the top to the bottom. Taking the union of any subset of these chains
gives a distinct subtree, so we have at least $2^{2^{\ell-1}}$
subtrees, which is exponential in the number of vertices.
}

\section{Stars}\label{sec:stars}

In this section, we present an algorithm that computes the exact optimum
of \ref{MESP}, when $C$ is an arrowhead matrix (i.e., 
when the support graph $G(C)$ is a star with center 1)
under an easily-checkable sufficient condition.

We define the \emph{arrowhead} matrix
 \[
 A(\alpha_1,\alpha,D) := \left( \begin{array}{cc}
 \alpha_1 &  \alpha^\top  \\
 \alpha & D\end{array} \right),
 \]
 with $\alpha_1 \in \mathbb{R}$, 
$\alpha:=(\alpha_2,\ldots,\alpha_n) \in \mathbb{R}^{n-1}$,
$d:=(d_2,\ldots,d_n)\in \mathbb{R}_+^{n-1}$,
% $x \neq \mathbf{0}$ 
 and $D:=\Diag(d) \in \mathbb{R}^{(n-1)\times (n-1)}$;
see \cite{arrow}, for example.

 \begin{lemma}
$ A(\alpha_1,\alpha,D)  \succeq 0$ if and only if $\alpha_1 \geq \sum_{i=2}^n{\frac{\alpha_i^2}{d_i}}$.
 \end{lemma}
 
 \begin{proof}
 By symmetric row/column scaling, it is easy to see that 
  $A(\alpha_1,\alpha,D)\succeq 0$
 if and only if   $A(\alpha_1,\tilde \alpha,\alpha_1 I)\succeq 0$,
 where $\tilde \alpha_i := \alpha_i \sqrt{\frac{\alpha_1}{d_i}}$,
for $i=2,\ldots, n$.
 From \cite{AlizGold}, we have that 
 $ A(\alpha_1,\tilde \alpha,\alpha_1 I)  \succeq 0$ if and only if
  $\alpha_1 \geq \|\tilde \alpha\|$ --- this is how the ``ice-cream cone'' constraint is 
 typically modeled as a semidefinite-programming  constraint.
  Now plugging in the definition of $\tilde \alpha$ and simplifying,
  we obtain our result.
 \end{proof}
 
%  In Lemma 4 we show that it is easy to check if an arrowhead is PSD by verifying the inequality $\alpha \geq \sum_{i=1}^{n-1}{\frac{x_i^2}{d_i}}$. 

We consider now \ref{MESP}$(A(\alpha_1,\alpha,D),s)$, where we assume that 
$\alpha_1 \geq \sum_{i=2}^n{\frac{\alpha_i^2}{d_i}}$, so that $ A(\alpha_1,\alpha,D)  \succeq 0$. 
We will employ a greedy algorithm to solve\break \ref{MESP}$(A(\alpha_1,\alpha,D),s)$, but we do not simply apply such an algorithm directly. 
Rather, we will branch on element 1, and then apply a greedy algorithm to the two subproblems,
selecting the best solution so found.

Our greedy algorithm is a manifestation of a generic greedy maximization algorithm
for the more general problem: Given $f:2^N \mapsto \mathbb{R}$, 
find an $s$ element set $S\subset N$ with maximum value of $f(S)$. 

\addtocounter{algocf}{-1}

\medskip 
\begin{algorithm}[H]
\SetAlgoLined
\KwIn{$f:2^N \mapsto \mathbb{R}$; $1<s<n$;}
\KwOut{$S$;}
$S:=\emptyset$\;
\While{$|S|<s$}{
let $j^*:=\argmax\big\{f(S+j ~:~ j \in N\setminus S \big\}$\;
let $S:=S+j$\;
}
\caption{Generic greedy for $\max\left\{f(S) ~:~ |S|=s\right\}$}\label{alg:generic}
\end{algorithm}
\medskip

We let $S_0=\emptyset,S_1,\ldots,S_s$, be the sequence of iterates $S$
that are created, iteration by iteration, in Algorithm \ref{alg:generic}. 
We have the following nice property that we can exploit. 

\begin{lemma}\label{lem:greedy}
Suppose that $f:2^N \mapsto \mathbb{R}$ satisfies the following stability property:
 $f(T\cup\{i\}) \leq f(T\cup\{j\})$ for all $T\subset N\setminus\{i,j\}$,
whenever distinct $i,j$ in $N$ satisfy  $f(i)\leq f(j)$.
Then the iterates of Algorithm \ref{alg:generic} are maximizing sets for each cardinality. 
\end{lemma}

\begin{proof}
Without loss of generality, we assume that $f(1)\geq f(2)\geq \cdots \geq f(n)$.
Under our hypothesis, the sequence of sets $S_k:=\{1,2,\ldots,k\}$, $k=0,1,2,\ldots,s$,
is a valid sequence of sets to be produced by the greedy algorithm. 
By way of a proof by contradiction, suppose that some $S_k$, with $k\geq 2$, is
not a maximizing set of its cardinality.  
Among all maximizing solutions of cardinality $k$, let $S^*$ be
a maximizer that has the maximum number of elements in common
with $S_k$. We have $f(S^*)>f(S_k)$, and so
we can choose $j\in S^*\setminus S_k$ and $i\in S_k \setminus S^*$. Clearly,
by how $S_k$ is chosen and by the hypothesis of the lemma, we have $f(j)\leq f(i)$. 
But now the hypothesis of the lemma give us that $f((S^*\setminus\{j\})\cup\{j\})\geq 
f((S^*\setminus\{j\})\cup\{i\})$. Therefore $S^*\setminus\{j\}\cup\{i\}$ is also optimal,
but it has more elements in common with $S_k$ than $S^*$ does --- a contradiction. % i don't see why having more elements in common is a contradiction
\end{proof}

In fact, we are interested in \ref{MESP}$(C,s)$, whereupon 
Algorithm \ref{alg:generic} particularizes as follows:

\medskip 
\begin{algorithm}[H]
\SetAlgoLined
\KwIn{$C\in\mathbb{S}^n_+$; $1<s<n$;}
\KwOut{$S$;}
$S:=\emptyset$\;
\While{$|S|<s$}{
let $j^*:=\argmax\big\{\ldet C[S+j,S+j] ~:~ j \in N\setminus S \big\}$\;
let $S:=S+j$\;
}
\caption{Greedy for \ref{MESP}$(C,s)$}\label{alg:greedy}
\end{algorithm}
\medskip

\begin{remark}
In fact, using the Schur complement of $C[S,S]$ in $C[S+j,S+j]$, we have
\[
\ldet C[S+j,S+j] = \ldet C[S,S] + \log \left( 
C[j,j]-C[j,S]\left(C[S,S]  \right)^{-1} C[S,j]
\right).
\]
So, to calculate $j^*$ in Algorithm \ref{alg:greedy}, we simply find the largest 
diagonal element from the Schur complement of $C[S,S]$  in $C$: 
\[
C[N\setminus S,N\setminus S]-C[N\setminus S,S]\left(C[S,S]  \right)^{-1} C[S,N\setminus S].
\]
\end{remark}
\bigskip

\vbox{
\noindent {\bf Branching on element 1.}
\begin{itemize}
\item 
If $1$ is not in some optimal solution of \ref{MESP}$(A(\alpha_1,\alpha,D),s)$, 
then 
\[
z(A(\alpha_1,\alpha,D),s)=z(D,s)= \sum_{i=1}^s {\log d_{[i]}}\thinspace. 
\]
That is, if  $1$ is not in some optimal solution,
then such an optimal solution
contains $s$ values of $i$
(from $\{2,\ldots,n\}$) with largest $d_i$.
Or, to put it another way, an optimal solution is found by applying 
 Algorithm \ref{alg:greedy}  to
\ref{MESP}$(D,s)$.
\item Alternatively,
if $1$ is in some optimal solution of \ref{MESP}$(A(\alpha_1,\alpha,D),s)$, 
 then
\[
z(A(\alpha_1,\alpha,D),s)=\log \alpha_1 +\textstyle z(D-\frac{1}{\alpha_1}\alpha\alpha^\top,s-1).
\]
So, then
\[
z(A(\alpha_1,\alpha,D),s) = \max\left\{\sum_{i=1}^s {\log d_{[i]}},~ \log \alpha_1 +\textstyle z(D-\frac{1}{\alpha_1}\alpha\alpha^\top,s-1) \right\},
\]
and it remains only to calculate 
$\textstyle z(D-\frac{1}{\alpha_1}\alpha\alpha^\top,s-1)$.
%  \begin{algorithm}[H]
%     \caption{Arrowhead greedy}\label{alg:greedy}
%     \begin{algorithmic}[1]
%     	\State $C, n, s$
%     	\State $S_1 \gets \{1\}$
% 	\State  $C_{s} \gets D - \frac{1}{\alpha}{xx'}$
% 	\State $S_2 = \argmax_{\bar{S}}\{ \diag(C_s[\bar{S},\bar{S}]) : |\bar{S}| = s-1\}$
% 	\State $S_a = \{S_1,S_2\}$
% 	\If {$\texttt{prod}(\texttt{maxk}(\diag(C),s)) > \det(C[S_a,S_a]) $}
% 		\State $ S = \argmax_{\bar{S}}\{ \diag(C[\bar{S},\bar{S}]) : |\bar{S}| = s\}$
% 		\Else { $S = S_a$}
% 	\EndIf
% 	\State \Return $S$
%     \end{algorithmic}
% \end{algorithm}
%
% %add explanation on connection of greedy alg to Schur Complement?
%
% Algorithm \ref{alg:greedy} compares two cases, first it evaluates a solution when $1 \in S$ and then it compares it to the optimal solution when $1 \notin S$. The first solution is made up the indices that have the largest diagonal elements in the Schur complement when $1 \in S$. In theorem \ref{thm:star} we show that for $\alpha_0 > \hat{\alpha}$ the order of elements in the Schur complement does not change for any $S$, this proves that one of the two cases evaluated by Algorithm \ref{alg:greedy} is the optimal solution.
So in what follows, we employ Algorithm \ref{alg:greedy},
although we will apply it to 
   \ref{MESP}$\textstyle(D-\frac{1}{\alpha_1}\alpha\alpha^\top,s-1)$.
\end{itemize}
}

Now, we are prepared to describe our sufficient condition
for Algorithm \ref{alg:greedy} to 
correctly solve \ref{MESP}$\textstyle(D-\frac{1}{\alpha_1}\alpha\alpha^\top,s-1)$).
What we will show is that as we take successive Schur complements in
$\textstyle D-\frac{1}{\alpha_1}\alpha\alpha^\top$, the ordering of the remaining 
diagonal elements is not affected.

\begin{definition}
 Let $r_k := \frac{\alpha_k^2}{d_k}$ for $2\leq k 
 \leq n$, and we choose a (sorting) bijection $\pi:\{1,\ldots,n-1\}\rightarrow \{2,\ldots,n\}$, 
 such that $r_{\pi(1)} \geq  r_{\pi(2)} \geq \cdots r_{\pi(n-1)}$\thinspace.
 Let $\Phi(s-1) := \{ \pi(k)  ~:~ 1\leq k \leq s-1 \}$.  
%  let $r_{[j]}$ denote the $j$-th greatest $r_i$, for $1\leq j \leq s-1$,
%  and let $S_{s-1}$ denote the set of 
%  let $\pi$ be a permutation of $\{2,...,n\}$ such that $r_{\pi(2)} \geq  r_{\pi(1)} \geq \cdots r_{\pi(n)}$.
\end{definition}

\noindent So, $\pi$ sorts the ratios $r_i$\thinspace, and then $\Phi$ selects the original indices of the $s-1$ 
largest. 

 \begin{theorem}
 \label{thm:star}
Let
 \[
 \hat{\alpha}_1 :=
% \sum_{j=1}^{s-1}{r_{[j]}} 
 \sum_{k\in \Phi(s-1) } r_k
 + \max_{i,j \in \{2,\ldots,n\} \setminus \Phi(s-1) }\left\{\frac{\alpha_i^2 - \alpha_j^2}{d_i - d_j} ~:~ d_i > d_j,~ \alpha_i^2 > \alpha_j^2\right\}.
 \]
If $  A(\alpha_1,\alpha,D) \succeq 0$, 
   $\alpha_1 \geq \hat{\alpha}_1$\thinspace,  then  Algorithm \ref{alg:greedy} produces an optimal
   solution of
   \ref{MESP}$\textstyle(D-\frac{1}{\alpha_1}\alpha\alpha^\top,s-1)$.
 \end{theorem}
 
  \begin{proof}
% Let $S\ni 1 $, and for convenience, 
 Let $\bar{D}:=D-\frac{1}{\alpha_1}\alpha\alpha^\top$. 
 %, and let $\bar{N}:=N\setminus\{1\}$. 
% Now, 
\emph{If the Schur complement of $\bar{D}[T,T]$ in $\bar{D}$
has the same ordering of its diagonal elements as $\bar{D}[{N}\setminus T,{N}\setminus T]$,
for all $T\subset N\setminus\{1\}$}, then Algorithm \ref{alg:greedy} applied to 
  \ref{MESP}$\textstyle(\bar{D},s-1)$  
  will choose $s-1$ values of $i$ corresponding to the
$s-1$ greatest diagonal elements of $\bar{D}$, and that will be optimal for 
\ref{MESP}$\textstyle(\bar{D},s-1)$.
By Lemma \ref{lem:greedy} Algorithm \ref{alg:generic}, and hence Algorithm \ref{alg:greedy}, generates the optimal solution.
%The reason that it is optimal is that, after element 1 is taken, the preference ordering of elements in ${N}\setminus\{1\}$ is independent of what elements have already been taken by the greedy algorithm. 
%Alternative explanation:
%We use the fact that at each step the greedy algorithm adds the index $j$ to $S$ that satisfies $\ldet(C[S+j,S+j]) \geq \ldet(C[S+i,S+i])$ for $i \in N \setminus S$. If we can guarantee that the ordering of all elements in $N\setminus S$ is preserved for any $S \cup {1}$ then the greedy algorithm produces the optimal solution. 

So, it remains to demonstrate that if $\alpha_1 \geq \hat{\alpha}_1$\thinspace,
then the Schur complement of $\bar{D}[T,T]$ in $\bar{D}$
has the same ordering of its diagonal elements as $\bar{D}[{N}\setminus T,{N}\setminus T]$,
for all $T\subset N\setminus\{1\}$.  In what follows, it is notationally easier to
work with $D$ rather than $\bar{D}$, so we consider  $1\in T\subset N$ rather than
$T\subset N\setminus\{1\}$. 
%Let $C_S\setminus\{1\} = C[N\setminus S,N\setminus S]-C[N \setminus S,S]\left(C[S,S]  \right)^{-1} C[S,N \setminus S]$  and 

For $1\in T\subset N$, with $|T|=t$,
the Schur complement of $A[T,T]$ in $A:=A(\alpha_1,\alpha,D)$ is
\begin{align*}
&A[N\setminus T,N\setminus T]- A[N\setminus T,T]\left( A[T,T]\right)^{-1} A[T,N\setminus T]\\
&\quad = D[N\setminus T,N\setminus T] - \left(\alpha[N\setminus T],0_{(n-t)\times t} \right) \left( A[T,T]\right)^{-1}  \left( \alpha[N\setminus T],0_{(n-t)\times t} \right)^\top\\
&\quad = D[N\setminus T,N\setminus T] - \left( A[T,T]\right)^{-1}_{11} \alpha[N\setminus T] \alpha[N\setminus T]^\top.
\end{align*}
Now, for $i\in N\setminus T$, the diagonal entry indexed by $i$ of this Schur complement is
\begin{align*}
d_{ii} - \left( \frac{1}{\alpha_1 - \sum_{k\in T\setminus\{1\}}{\frac{\alpha_k^2}{d_{k}}}}\right) \alpha_{i}^2~,
\end{align*}
where we have extracted $\left( A[T,T]\right)^{-1}_{11}$
using the standard block-matrix inverse formula on 
\[
A[T,T] = \left(
  \begin{array}{c|c}
\alpha_1 & \alpha[T\setminus\{1\}]^\top \\[2pt]
\hline 
\alpha[T\setminus\{1\}] & D[T\setminus\{1\},T\setminus\{1\}] \vphantom{x^{\strut}} \\
  \end{array}
\right).
\]

If $T=\{1\}$, then diagonal element $i$ of the Schur complement is $d_i-\frac{1}{\alpha_1}\alpha_i^2$\thinspace
So we want to demonstrate that, for all $T\ni 1$ and $i,j\in T\setminus N$, if
\begin{equation}\tag{$*$}\label{star}
d_i-\frac{1}{\alpha_1}\alpha_i^2
~ \geq ~
d_j-\frac{1}{\alpha_1}\alpha_j^2~,
\end{equation}
then 
\begin{equation}\tag{$**$}\label{starstar}
    d_i - \dfrac{\alpha_i^2}{\alpha_1 - \sum_{k \in T\setminus\{1\}}{\frac{\alpha_k^2}{d_k}}} ~ \geq ~ d_j - \dfrac{\alpha_j^2}{\alpha_1 - \sum_{k \in T\setminus\{1\}}{\frac{\alpha_k^2}{d_k}}}~.
\end{equation}

Note that \eqref{star} implies that either $d_i \geq d_j$ and $\alpha_i^2 \leq \alpha_j^2$\thinspace,
in which case \eqref{starstar} is trivially true, or $d_i > d_j$ and $\alpha_i^2 > \alpha_j^2$\thinspace.
In this latter case, \eqref{starstar} reduces to 
\begin{equation*}
\alpha_1   \geq  \sum_{k \in T\setminus\{1\}}{\frac{\alpha_k^2}{d_k}} + \dfrac{\alpha_i^2 - \alpha_j^2}{d_i - d_j}~.
\end{equation*}
The result now follows.  \end{proof}

% The ordering of the diagonal elements is evaluated by observing the following statement,

% \begin{align*}
%     d_i - \dfrac{\alpha_i^2}{\alpha_1 - \sum_{k \in S_{-1}}{\frac{\alpha_k^2}{d_k}}} &\geq  d_j - \dfrac{\alpha_j^2}{\alpha_1 - \sum_{k \in S_{-1}}{\frac{\alpha_k^2}{d_k}}} & i,j \notin S.
% \end{align*}
 
%  Note that the above statement is trivially true if $d_i \geq d_j$ and $\alpha_i^2 \leq \alpha_j^2$ and trivially false if $d_j > d_i$ and $\alpha_j^2 < \alpha_i^2$. In these trivial cases, the order of the diagonal elements is preserved for all $\alpha_1$ where $A$ is PSD, so we are only concerned with the case when $d_i > d_j$ and $\alpha_i^2 > \alpha_j^2$ (w.l.o.g.).
 
%  \begin{align*}
%     d_i - \dfrac{\alpha_i^2}{\alpha_1 - \sum_{k \in S_{-1}}{\frac{\alpha_k^2}{d_k}}} &\geq  d_j - \dfrac{\alpha_j^2}{\alpha_1 - \sum_{k \in S_{-1}}{\frac{\alpha_k^2}{d_k}}} & i,j \notin S, d_i > d_j \text{ and } \alpha_i^2 > \alpha_j^2  \\
%     % d_i - d_j &\geq  \dfrac{\alpha_i^2 - \alpha_j^2}{\alpha_1 - \sum_{k \in S}{\frac{\alpha_k^2}{d_k}}} & \\
%     % \alpha_1 - \sum_{k \in S}{\frac{\alpha_k^2}{d_k}} & \geq  \dfrac{\alpha_i^2 - \alpha_j^2}{d_i - d_j}&\\
%     \alpha_1  & \geq  \sum_{k \in S_{-1}}{\frac{\alpha_k^2}{d_k}} + \dfrac{\alpha_i^2 - \alpha_j^2}{d_i - d_j} & \tag{1}
% \end{align*}
 
% Finally we have that for any $\alpha_1$ that satisfies inequality (1) for all $S$ the diagonal elements of the Schur complement are preserved and that 
% $\hat{\alpha_1}$ satisfies the inequality in (1) for all $S$.

\begin{example}
If the sufficient condition of Theorem \ref{thm:star} does not hold, then indeed the greedy algorithm may not find an optimum.
%Can you insert an example where  $  A(\alpha_1,\alpha,D) \succeq 0$, $\alpha_1 < \hat{\alpha}_1$ and the greedy algorithm fails?  I think that you had one for $n=5$. 
For example,
Let
\[ 
A = \left( \begin{array}{ccccc}
    12 & 3.5 & 1.9 & 0.04  & 4.9\\
    3.5  & 4 & 0 & 0 & 0\\
    1.9  &  0 &  3 & 0 & 0 \\
    0.04  &  0 & 0  & 2.5 & 0\\
    4.9  & 0 &  0 & 0 & 5
\end{array} \right)
\]
and take $s=3$.
It is easy to check that
\begin{align*}
    \alpha_1 = 12 &< \hat{\alpha}= 15.0831 = \frac{4.9^2}{5} + \frac{3.5^2}{4} + \frac{1.9^2 - 0.04^2}{3-2.5}. \\
%     & = \sum_{k\in S_{s-1}} r_k
%  + \max_{i,j \in \{2,\ldots,n\} \setminus S_{s-1}}\left\{\frac{\alpha_i^2 - \alpha_j^2}{d_i - d_j} ~:~ d_i > d_j,~ \alpha_i^2 > \alpha_j^2\right\} = \hat{\alpha}_1.
\end{align*}
The diagonal elements of the Schur complements of $A[T,T]$ 
for $T:=\{1\}$ and $T:= \{1,5\}$ respectively are

%matrix version
% \[\left[\begin{array}{cccc}
%   4 - \frac{3.5^2}{12}  & .   & .   & .\\
%   .   & 3 - \frac{1.9^2}{12}  & .  & . \\
%   .   & .   & 2.5 - \frac{0.04^2}{12} & .\\
%   .   & .   &   .  & 5 - \frac{4.9^2}{12}
% \end{array}\right] 
% %
% = \left[\begin{array}{cccc}
%   2.9792  & .   & .   & .\\
%   .   &  2.6992   & .  & . \\
%   .   & .   & 2.4999 & .\\
%   .   & .   &   .  & 2.9992
% \end{array}\right] 
% \]

\[
% \left(\begin{array}{c}
%   4 - \frac{3.5^2}{12} \\
%   3 - \frac{1.9^2}{12}  \\
%   2.5 - \frac{0.04^2}{12} \\
%   5 - \frac{4.9^2}{12}
% \end{array}\right)
\bordermatrix{
&&\cr
2 &  4 - \frac{3.5^2}{12}\hspace{-10pt} \cr
3 &  3 - \frac{1.9^2}{12}\hspace{-10pt} \cr
4 &   2.5 - \frac{0.04^2}{12}\hspace{-10pt}\cr
5 &  5 - \frac{4.9^2}{12}\hspace{-10pt}
}
\!=\! \left(\begin{array}{c}
  2.9792  \\
  2.6992   \\
  2.4999 \\
  2.9992
\end{array}\right) 
\mbox{ and }\quad
% \implies 5 \in S
% %
% \implies
%
% \left(\begin{array}{c}
%   4 - \frac{3.5^2}{12 - \frac{4.9^2}{5}} \\
%   3 - \frac{1.9^2}{12 - \frac{4.9^2}{5}}  \\
%   2.5 - \frac{0.04^2}{12 - \frac{4.9^2}{5}} 
% \end{array}\right)
%
\bordermatrix{
&&\cr
2&  4 - \frac{3.5^2}{12 - 4.9^2/5}\hspace{-10pt} \cr
3&  3 - \frac{1.9^2}{12 -  4.9^2/5}\hspace{-10pt}   \cr
4&  2.5 - \frac{0.04^2}{12 -  4.9^2/5}\hspace{-10pt}  
}
\!=\! \left(\begin{array}{c}
  2.2981  \\
  2.4985   \\
  2.4998 
\end{array}\right) .
\]
% \[\left[\begin{array}{c}
%   4 - \frac{3.5^2}{12 - \frac{4.9^2}{5}} \\
%   3 - \frac{1.9^2}{12 - \frac{4.9^2}{5}}  \\
%   2.5 - \frac{0.04^2}{12 - \frac{4.9^2}{5}} 
% \end{array}\right] 
% %
% = \left[\begin{array}{c}
%   2.2981  \\
%   2.4985   \\
%   2.4998 
% \end{array}\right] 
% \].
The ordering implied by the first Schur complement is $5,2,3,4$ 
and for the second it is $4,3,2$. The greedy solution
appends 5 and then 4 to $\{1\}$,
but the optimal solution turns out to be $S^* = \{1,2,3\}$. 
%and we have that $89.967 = \det(A(S_g,S_g)) < \det(A(S^*,S^*)) = 92.81$. % can be made a little more clear by making it obvious that some ratio not the max is sufficient to show alpha doesn't meet the condition
\end{example}

\section{Tridiagonal masks}\label{sec:tridiagM}

% Furthermore, upper bounds for
% $z(C^{-1},n-s)$  yield upper bounds for $z(C,s)$, shifting by $\ldet C$,
% and upper bounds (under this complementation)
% are \emph{not} always equivalent. 

It may well be that neither $G(C)$ nor $G(C^{-1})$ is 
a spider with a constant number of legs.
Even then, we can use our dynamic-programming algorithm to
get a bound on $z(C,s)$, because $\ldet C[S,S] \leq \ldet (C\circ M)[S,S]$,
for all $S\subset N$.
It is evident that
when $M$ is sparse and $C$ is fully dense,
$G(C\circ M)=G(M)$, and so when $G(M)$ is a spider with
a constant number of legs, we can apply
our dynamic-programming algorithm  to 
$C\circ M$ to efficiently  get an upper bound on $z(C,s)$. 
 Furthermore, upper bounds for
$z(C^{-1},n-s)$  yield upper bounds for $z(C,s)$, shifting by $\ldet C$,
and upper bounds (under this complementation)
are \emph{not} always equivalent. So we can as well 
profitably apply masking to $C^{-1}$.

For simplicity of exposition
and because we have developed the theory in more detail for tridiagonal
masks $M$ (i.e., when $G(M)$ is a collection of disjoint paths), we confine
our attention to tridiagonal masks $M$.
To set some notation, a \emph{tridiagonal mask}  $M\in \mathbb{R}^{n\times n}$ has the form
\[
M:=\left(
  \begin{array}{ccccc}
    1 & \mu_1 &  &  &  \\
    \mu_1 & 1 & \mu_2 &  &  \\
     & \mu_2 & \ddots & \ddots &  \\
     &  & \ddots & \ddots & \mu_{n-1} \\
     &  &  & \mu_{n-1} & 1 \\
  \end{array}
\right),
\]
with $M_{ij}:=0$ when $|i-j|>1$.

For $M$  to be a
 mask, we need it to be positive semidefinite. 
 For example, we can check that
if we have all $\mu_i:=1$, then $M$ is not positive semidefinite (for all $n>2$); but if we have all $|\mu_i|\leq \frac{1}{2}$,
 then $M$ is diagonally dominant and hence positive semidefinite for all $n$. A \emph{$\frac{1}{2}$-mask}
is such an $M$ with all $|\mu_i|= \frac{1}{2}$.
 In fact, we can do better than this, in the sense that we can increase \emph{some} entries from 
 $\frac{1}{2}$, which seems empirically 
 to be valuable for getting better 
 bounds. For example: for $n=2$, we can set $\mu_1=1$; and for $n=3$, we could set $\mu_1=\frac{1}{2}$
 and $\mu_2=\sqrt{\frac{3}{4}}$.
 For practical purposes, we will limit the number of pairs of symmetric entries that are increased from $\frac{1}{2}$ to two, and we want to
 see how much we can increase such entries up from  $\frac{1}{2}$.
 Clearly an upper bound on the maximum value of each $\mu_i$ is 1, because 
$
\begin{psmallmatrix*}[c]1 & \mu_i\\ \mu_i & 1\end{psmallmatrix*}
$ 
is always a principal submatrix.
 
For $1\leq p<q < n$, and $a,b\in \mathbb{R}$, let 
$M:=M(n,p,a,q,b)$ be an order-$n$ matrix that differs from the 
order-$n$ $\frac{1}{2}$-mask in that $M_{p+1,p}=M_{p,p+1}=\mu_{p}:=a$ and
$M_{q+1,q}=M_{q,q+1}=\mu_{q}:=b$. We will also denote $M(n,p,a) := M(n,p,a,p,a)$ and $M_{1/2}(n) := M(n,p,1/2)$  $\forall 1 \leq p \leq n-1$.

% \begin{figure}
%     \centering
%     \includegraphics[width =75mm]{maxa.png}
%     \caption{Maximum $a$ given $p=1$}
%     \label{fig:a^*}
% \end{figure}

The following lemma has a somewhat lengthy technical proof, which can be
found in Appendix 1. 

\begin{lemma}
\label{lem:det}
For $1\leq p<q < n$, and $a,b\in \mathbb{R}$, 
\begin{align*}
\det M(n,p,a,q,b) = 
&\dfrac{1}{2^{n}} (n-q+1)\Big( {(p+1)(q-p+1) - 4a^2p(q-p)}\Big) \\
&\quad -\dfrac{1}{2^{n}} (n-q)4b^2\left((p+1)  (q-p)- 4a^2p(q-p-1)\right) .
\end{align*}
\end{lemma}

% \[
% 2^{-n} \left((n-q+1) \left((p+1) (-p+q+1)-4 a^2 p (q-p)\right)-4 b^2 (n-q)
%   \left((p+1) (q-p)-4 a^2 p (-p+q-1)\right)\right)
% \]

From this, we can see how large we can make $a$, when $b$ is held at $1/2$.
In particular, we will see that if $b$ is held at  $1/2$, we
can make $a$ at least $\sqrt{2}/2$ by positioning $a$ at the first off-diagonal pair.

\begin{proposition}\label{prop:astar}
The maximum value of $a$ such that $M(n,p,a) \succeq 0$ is
\vskip-5pt
\[ 
%a^*(n,p) := \sqrt{\dfrac{(p+1)(n-p+1)}{4p(n-p)}}.
a^*(n,p) := \frac{1}{2}\sqrt{\left(1+\frac{1}{p}\right)\left(1+\frac{1}{n-p}\right)}.
\]
\vskip-5pt
\noindent Furthermore, $a^*(n,p)$ is convex and decreasing in $n$ and $\lim_{n\to\infty} a^*(n,p) =   \frac{1}{2} \sqrt{1+1/p}$. 
In particular, $\displaystyle \max_{p ~:~ 1\leq p <n-1} ~\lim_{n\to\infty} a^*(n,p) =
\lim_{n\to\infty} a^*(n,1) = 
\sqrt{2}/2$.
\end{proposition}

\begin{proof}
%\color{blue}
Note that $M(n,p,\frac{1}{2})=M_{1/2}(n) \succeq 0$, because it is diagonally dominant.  Also, we have  that $\det M(n,p,\frac{1}{2}) > 0 $ (by  Lemma \ref{lem:det}).  Using Lemma \ref{lem:det}, we can 
see that 
\[
\det M(n,p,a)= 
\dfrac{1}{2^{n}} \Big( (p+1)(n-p+1) - 4a^2p(n-p)\Big) 
\]
is decreasing in $a$. Now, solving $\det M(n,p,a)=0$ for $a$, and carrying out some algebraic manipulations,
we obtain the formula for $a^*(n,p)$ stated in the result. 
What we can conclude, at this point, is that $a^*(n,p)$ is an upper bound on the 
maximum value of $a$ such that $M(n,p,a) \succeq 0$. 

To see that this upper bound is
in fact the true maximum, we need to verify that $M(n,p,a^*)\succeq 0$.
To do this, we will check that all leading principal submatrices of order $1\leq k <n$
have positive determinant. This, together with the nonnegativity of the determinant
of $M(n,p,a^*)$ (it is in fact 0) is sufficient to establish that $M(n,p,a^*)\succeq 0$ (see, for example, \cite[Exercise at the bottom of p. 404]{HJBook}).

Now, for $k\leq p$, the $k$-th leading principal submatrix is an order-$k$ $\frac{1}{2}$-mask, which has positive
determinant by Lemma \ref{lem:det}. For $p < k <n$, the 
$k$-th leading principal submatrix is $M(k,p,a^*)$.
After some algebraic manipulations, we arrive at 
\[
	\det M(k,p,a^{*})
	=  \frac{p+1}{2^{k}}
	(k-p)
	\left(
	\frac{1}{k-p} -	\frac{1}{n-p} 
	\right),
\]
which is clearly positive for $p < k <n$.

% \noindent
% We conclude that $\det(M(k,r,a^{*},k,\frac{1}{2})) > 0$ because 

% \begin{align*}
% (k-p+1)(p+1) &> \frac{(k-p)p(n-p+1)(p+1)}{p(n-p)} & \\
% \frac{(k-p+1)(p+1)}{(k-p)p} &> \frac{(n-p+1)(p+1)}{(n-p)p} & \text{ by }	n & > k 
% \end{align*}

% \noindent
% This is true while

% \begin{align*}
% 	p & \geq r \tag{by definition}\\
% 	n-p & \geq k-r \geq k- \max{r} \tag{by definition}
% \end{align*}
% It is easy to see that $a^*$ is decreasing in $n$ and convex in $n$ where $n>p\geq 1 $. The result $\lim_{n\to\infty} a^*(n,p) =   \frac{1}{2} \sqrt{1+1/p}$ can be seen after multiplying out the $p+1 $ factor in the numerator.
%hessa
Thus, we have established that $a^*(n,p)$ is  the
maximum value of $a$  such that $M(n,p,a) \succeq 0$. 
The rest of the result is easy to verify.
\end{proof}
%\noindent (Note that with respect to Proposition \ref{prop:astar}, when $b=1/2$, $q$ is arbitrary.)

Knowing the possible increases of $a$ from $1/2$, we can determine the maximum of $b$ (given $a$) and the general behavior of the 
maximum as we vary $p$ and $q$.

\begin{theorem}\label{thm:bmax}
For $1\leq p<q < n$ and $a\in [\frac{1}{2},a^*(n,p)]$, we have that 
$M(n,p,a,q,b) \succeq 0$ if and only if $b\in[\frac{1}{2},b^*(n,p,a,q)]$, where
\[
b^*(n,p,a,q): = \frac{1}{2}\sqrt{\dfrac{(n-q+1)\Big((p+1)(q-p+1) - 4a^2p(q-p)\Big)}{
(n-q)\Big((p+1)(q-p) - 4a^2p(q-p-1) \Big)}}. 
\]
\end{theorem}

\begin{proof}
It is easy to verify, from Lemma \ref{lem:det}, that the formula given for $b^*(n,p,a,q)$
is the unique nonnegative solution of $\det M(n,p,a,q,b)=0$. 

We claim that $\det M(n,p,a,q,b)$
is decreasing in $b$ when $a\in [\frac{1}{2},a^*(n,p)]$ (i.e., the
the range of $a$ that ensures $M(n,p,a) \succeq 0$).
It is evident, from Lemma \ref{lem:det}, that this is when
\[
(p+1)(q-p) - 4a^2p(q-p-1) > 0,
\]
or simply when
\[
%a < \tilde{a} := \frac{1}{2}\sqrt{\frac{(p+1)(q-p)}{p(q-p-1)}}.
a < \tilde{a}(q,p) := \frac{1}{2}\sqrt{\left( 1+\frac{1}{p}\right) 
  \left(1+\frac{1}{(q-1)-p} \right)}.
\]
Notice that with $q<n$, we clearly have $\tilde{a}(q,p)>a^*(n,p)$, and 
therefore   $\det M(n,p,a,q,b)$
is decreasing in $b$ when $a\in [\frac{1}{2},a^*(n,p)]$.

We can conclude that for $a\in [\frac{1}{2},a^*(n,p)]$, 
the maximum value of $b$ such that $\det(M(n,p,a,q,b))\geq 0$ is 
$b^*(n,p,a,q)$.

% . After some algebraic manipulations using the formula for the determinant in Lemma \ref{lem:det} we have that,

% \begin{align}
% \label{eq:bstar}
%   b^{*2} = \dfrac{(n-q+1)\left((q-p+1)(p+1) - (q-p)4pa^2\right)}{  4(n-q)\left((q-p)(p+1)- (q-p-1)4pa^2\right)}
% \end{align}

Now, it only remains to demonstrate  that $M(n,p,a,q,b)\succeq 0$, for all $b\in[\frac{1}{2},b^*(n,p,a,q)]$.
It suffices to demonstrate that all proper leading principal submatrices of $M(n,p,a,q,b^*)$ 
have a positive determinant. Because $\det M(n,p,a,q,b)$ is decreasing in $b$ (for the relevant $n$, $p$, $a$, $q$),
we will then have that all   proper leading principal submatrices of $M(n,p,a,q,b)$ 
have a positive determinant, for all $b\in[\frac{1}{2},b^*(n,p,a,q)]$.
This will then imply that $M(n,p,a,q,b)\succeq 0$, for all $b\in[\frac{1}{2},b^*(n,p,a,q)]$.

 From Proposition \ref{prop:astar}, we have that 
 principal submatrices that take the form $M_{1/2}(k)$ and $M(k,p,a)$ with $p < k < n$ are positive semidefinite, for $a\in [\frac{1}{2},a^*]$. So it remains to show that  $M(k,p,a,q,b^*)$ has a positive determinant, for $1\leq p < q < k <n $. %$M(k,q,b^*)$ and

% For $M(k,q,b^*)$ we will show that $\max\{a: M(k,q,a)\} \geq b^*(n,p,1/2,q) \geq b^*(n,p,a^*,q) $, which implies $ \det(M(k,q,b^*)) \geq \det(M(k,q,a)) \geq 0$.

% \begin{align*}
% a^{*2}(k,q) & = \dfrac{(k-q+1)(q+1)}{4q(k-q)} \\ %=    \dfrac{(k-q+1)(q+1)}{  4q(k-q)} \tag{For $a=1/2$}\\
%         %   & =   \dfrac{(n-q+1)(qp-p^2+p+q-p+1-qp+p^2)}{  4(n-q)(qp-p^2+q-p-qp+p^2+p)}\\
%         %     & =   \dfrac{(n-q+1)\left((q-p+1)(p+1) - (q-p)4p(1/2)^2\right)}{  4(n-q)\left((q-p)(p+1)- (q-p-1)4p(1/2)^2\right)}\\
%              & = b^{*2}(k,p,1/2,q)\\
%             & \geq  b^{*2}(k,p,a,q)\tag{For $a \geq 1/2$} .
%  \end{align*}
  
% Note that for the last inequality to be true $b^*$ must be decreasing in $a$. This can be confirmed by inspecting the derivative of $b^{*2}$ w.r.t $a$, for $1 \leq p < q < n$ and $\frac{1}{2} \leq a \leq a^*$.

% $$\frac{\partial b^{*2}}{\partial a} =  -\dfrac{2(n-q+1)}{(n-q)}\frac{a p (p+1)}{ \left(\left( 4a^2-1\right) p (p-q+1)+q\right)^2} < 0.$$

Plugging in the formula for $b^*(n,p,a,q)$ into the formula for $\det M(k,p,a,q,b^*)$, 
after some simplifications we have
\begin{align*}
 & \frac{1}{2^k} \underbrace{\left(
 (p+1)(q-p+1) - 4a^2p(q-p)
 \right)}_{(*)}
 \underbrace{\left(
 (k-q+1)(n-q)-(n-q+1)(k-q)
 \right)}_{(**)}.
\end{align*}
It is easy to check that $(*)$ is positive because $a < \tilde{a}(q,p)$,
and $(** )$ is positive because $(k-q+1)/(k-q) > (n-q+1)/(n-q)$, for $1 < q < k <n $.
\end{proof}

%   \begin{align*}
% & \det M(k,p,a,q,b^*)  = \frac{1}{2^{k}} (k-q+1)(n-q)\left( {(q-p+1)(p+1) - (q-p)4pa^{2}}\right)\\
% &  \qquad -\frac{1}{2^{k}}(n-q+1)(k-q)\big((q-p+1)(p+1) - (q-p)4pa^{2}\big) > 0.
%   \end{align*}
%   This statement can be further simplified to
%   \begin{align*}
%   \dfrac{k-q+1}{k-q} > \dfrac{n-q+1}{n-q}
%   \end{align*}
%   which is true for $k < n$.
%   {\color{blue} Can you write this in a way that makes it clear that it is positive?}

With the next two results (proofs deferred to Appendix 1), 
we see that insofar as maximizing the value of $b$ in  $M(n,p,a,q,b)$,
if $n,q,a$ are fixed, we should set $p=1$;  and then with $p=1$, we should set $q=n-1$. 

\begin{theorem} 
\label{thm:bp1}
For each  $1\leq p<q<n$,  $a\in (\frac{1}{2},a^*(n,p)]$, and $n \geq 4$, 
we have that 
$b^*(n,1,a,q) \geq b^*(n,p,a,q)$.
\end{theorem}

\begin{theorem}\label{thm:end}
$b^*(n,1,a,q)$ is maximized over $q$ for $1<q< n$
at $q^*(n,1,a)=n-1$.  
\end{theorem}

Note that by Theorems \ref{thm:bp1} and \ref{thm:end}, $b^*(n,p,a,n-1) \geq b^*(n,p,a,q)$.
Symmetrically, corresponding to the last four results,
we can establish: (i) a formula for the maximum value of $b$ when $a=1/2$, (ii) 
a formula for the maximum value $a^*$ of $a$, (iii) that the
$a^*$ is increasing in $q$, and (iv)  that  $a^*$ is maximized over $p$ at $p^*=1$.

\section{Local search on tridiagonal masks}\label{sec:masklocalsearch} 

Our goal is to do a local search in the space of tridiagonal masks, so
as to get a (tridiagonal) mask $M$ that leads to a 
low value of $f(C\circ M,s)$ for some upper-bounding method $f(\cdot,s)$
applied to \ref{MESP}. 
Besides exact calculation by dynamic programming of $z(C\circ M)$, we
work with the following additional bounding methods:
\begin{itemize}
    \item 
The (masked) \emph{spectral bound} for $z(C\circ M,s)$:
\[
\textstyle
\sum_{\ell=1}^s \log \lambda_\ell(C\circ M),
\]
where $\lambda_\ell(\cdot)$ denotes the $\ell$-th greatest eigenvalue
When $M=I$, the masked spectral bound is the \emph{diagonal bound}:
the sum of the logs of the $s$ biggest diagonal components of $C$.
The spectral bound was first presented in  \cite{KLQ}, and its masked versions in \cite{HLW,AnstreicherLee_Masked,BurerLee}. 
\item
The (masked) \emph{linx bound} for $z(C\circ M,s)$ is the optimal value of the following convex-optimization relaxation:
\begin{align*}
&\max~ \textstyle \frac{1}{2} \left(\ldet(\gamma (C\circ M)  \Diag(x)(C\circ M)+\Diag(\mathbf{e}-x))-s\log \gamma \right)\\[3pt] 
&\quad\mbox{s.t.}\quad \mathbf{e}^\top x=s,~ 0\le x\le \mathbf{e},
\end{align*}
where $\gamma>0$ is a scaling parameter that must be judiciously selected,
and $\mathbf{e}$ is an all-ones vector. 
The linx bound was first presented in  \cite{Kurt_linx},
and its optimal scaling was studied in \cite{Kurt_linx,Mixing}.
\end{itemize}

Note that $z(C,s)$ is permutation invariant. That is,
if $\Pi$ is an $n\times n$ permutation matrix,
then $z(C,s)=z(\Pi C\Pi^\top,s)$.
Moreover the value of most  bounds applied to
$C$ and $\Pi C\Pi^\top$ are identical.
But these observations are \emph{not} at all generally true for $C\circ M$
compared to $(\Pi C\Pi^\top) \circ M$ (unless $M$ is permutation invariant, for example $M=I_n$ or $M=J_n$),
and so we can try to optimize the bound produced by a bounding method,
even for a fixed mask $M$, by varying the permutation matrix $\Pi$.
As a first phase, we start with $M$ equal to the $\frac{1}{2}$-mask of order $n$,
and we do a local search, using minimization of the spectral bound as a criterion, 
applied to $(\Pi C\Pi^\top) \circ M$. Each permutation matrix $\Pi$ 
is in one-to-one correspondence with a permutation $\pi$ of $N$
(i.e., each of the $n!$ permutations
$\pi$ is an ordered set $(\pi_1,\pi_2,\ldots,\pi_n)^\top$ of the $n$ distinct elements of $N=\{1,2,.\ldots,n\}$, so that $\Pi(1,2,\ldots,n)^\top = (\pi_1,\pi_2,\ldots,\pi_n)^\top)$). 
Our local moves correspond to choosing an $i,j$ with $1\leq i < j \leq n$,
and then replacing $\pi_i,\pi_{i+1},\ldots,\pi_{j-1},\pi_j$ with
$\pi_j,\pi_{j-1},\ldots,\pi_{i+1},\pi_i$. We do a best-improvement local search;
upon convergence, we replace $C$ with $\Pi C\Pi^\top$. 

%this paragraph makes the moves seem more complicated than they are
%$M=\Diag(M_1[S_1,S_1],M_2[S_2,S_2], \ldots, M_p[S_p,S_p])$
Next, we proceed to our second phase.
As a general step, we have a \emph{blocked tridiagonal mask}
$M:=\Diag(M[S_1,S_1],M[S_2,S_2], \ldots, M[S_p,S_p])$,
with $p\geq 1$, where the $S_i$ partition $N$, each $S_i$ 
is a nonempty  ordered contiguous set, each $M[S_i,S_i]$ is
a tridiagonal mask (of order $|S_i|$),
and for all $1\leq i<j\leq p$, 
all elements of $S_i$  are less than all elements of $S_j$. 
In one-to-one correspondence with the partitioning is its \emph{signature}
$(|S_1|,|S_2|,\ldots,|S_i|)$, a sequence of positive integers summing to $n$. 
Let  $(s_1,s_2,\ldots,s_p)$ be the current signature.
We do local search on the sets underlying the signature.
We want to choose a rich set of local moves that we
can reasonably search over for a best-improving move. 
The moves that we consider are:
\begin{itemize}
    \item \underline{Merge a pair of adjacent blocks}: If $p>1$, choose some $1\leq i<p$, and replace $s_i,s_{i+1}$ with $s_i+s_{i+1}$.
    %or name M_{|S_i|} where M_i is near 1/2 mask of order |S_i|
    \item \underline{Split a block}: If $p<n$, choose some $i$ with $s_i>1$, choose some $1\leq t <s_i$, and replace $s_i$ with $t,s_i-t$.
    \item \underline{Interchange}: If $p>1$ and not all $s_i$ are equal, 
    choose some $1\leq i < j \leq p$, with $s_i\not=s_j$, and swap $s_i$ with $s_j$.
\end{itemize}
The number of available moves  at each step of the local search is $\mathcal{O}(n^2)$, which 
we can reasonably search over.

Considering Theorem \ref{thm:end} (and its analogue for $a^*$),
for the purpose of our local search,
we will restrict our attention to $M[S_i,S_i]$
that are of the form $M(|S_i|,1,a,|S_i|-1,b)$.
So given an $S_i$, the only flexibility that we will take is
in choosing $a$ and $b$, only differing from an order-$|S_i|$
mask in the ``end pairs'':
\[
M[S_i,S_i]:=\left(
  \begin{array}{ccccccccc}
    1 & a & 0 & & & &  \\
    a & 1 & \frac{1}{2} & & & & & &\\
    0 & \frac{1}{2} & 1 &&  &  & & & \\
     &  &  \ddots  & \ddots & \ddots &  & & & \\
        &  &  &  \frac{1}{2} & 1 &\frac{1}{2} &  & &\\
   &&    &  &  \ddots  & \ddots & \ddots &  & \\   
  &&   & &  &  & 1 & \frac{1}{2} &  0\\
  &&  & &  & & \frac{1}{2}  & 1 & b \\
 &&   & &  &  & 0 & b & 1 \\
  \end{array}
\right),
\]

For any new block $M[S_i,S_i]=M(|S_i|,1,a,|S_i|-1,b)$ created in the local search over signatures,
we consider the following choices of $(a,b)$: $(\frac{1}{2},b)$ with $b$ maximized
and $(a,\frac{1}{2})$ with $a$ maximized. We calculate these maximum values using Theorem
\ref{thm:bmax}  (and its analogue for $a^*$).

We take these local search moves, and we perform a best-improvement local search on $f(C\circ M,s)$,
where $f$ is the spectral bound. Note that calculation of the spectral bound is quite fast and also
easily parallelizes across the blocks.
Then, we switch to the linx bound (generally a better bound, but slower
to calculate --- and not easily parallelized), continuing the local search.
Finally, we further continue the 
local search, using our dynamic-programming recursion to exactly calculate $z(C\circ M,s)$ at each step.
We note that  linx solves are much faster than dynamic-programming solves (nearly twice as fast in our experiments),
justifying the linx phase in our search.

\begin{table}[ht]
  \begin{center}
\begin{tabular}{c|c|c|c|c|c|c|c|c}
  $C$ & $n$ & $s$ & $\mbox{linx}(\frac{1}{2})$ & $\lambda(\frac{1}{2})$ & $\lambda(\Pi)$  & $\lambda(\Sigma)$ & $\mbox{linx}(\Sigma)$  & $\mbox{DP}(\Sigma)$      \\
  \hline
   $C_{63}^{-1}$ &  63 &  32 &  3.2512 & 3.2469  & 2.9521  & 2.7197  &  2.4828 & 2.4828 \\
   $C_{63}^{-1} + 0.125I$ & 63 & 32  & 3.1841  &  3.1801 & 2.8821  & 2.6584 & 2.4205  & 2.4205  \\
   $C_{63}^{-1} + 0.25I$  & 63  & 32  &  3.1182 & 3.1144  & 2.8302  & 2.5994  & 2.3649 &  2.3649  \\
   $C_{63}^{-1} + 0.375I$ & 63  &  32 &  3.0523 &  3.0488 & 2.7585  & 2.5404   & 2.3092 & 2.3092 \\
   $C_{63}^{-1} + 0.5I$ & 63 &  32 &  2.9891 & 2.9858  & 2.6992 & 2.4838 & 2.2558 & 2.2558 \\
   \hline
    $C_{1} $ & 100 &  50 &  5.7979 & 5.9629 & 5.4154  & 5.2334  & 4.8091  & 4.8091   \\
    $C_{2} $ & 100 &  50 & 5.6529  & 5.8427  & 5.1636  & 4.8221  & 4.3299  &  4.3299  \\
    $C_{3} $ & 100 &  50 & 6.0613  & 6.1358 & 5.6738  & 5.3606  & 4.8532  & 4.8532  \\
    $C_{4} $ & 100 &  50 & 5.2033 & 5.4311 & 4.8394  & 4.6112  & 4.2871  & 4.2871  \\
    $C_{5} $ & 100 &  50 & 4.835 & 4.8861 & 4.4351 & 4.1941 & 3.8815 &  3.8815 \\
    \hline
    $C_{6} $ & 100 &  50 & 0.511  & 1.6225 &  0.54995  & 0.4410  & 0.1263  &   0.1263  \\
    $C_{7} $ & 100 &  50 &  0.3095 & 1.3056  &  0.33 & 0.1901  & 0.1332  &   0.1332 \\
    $C_{8} $ & 100 &  50 &  0.3921 & 1.3292  & 0.4362  &  0.3298 & 0.2312  &  0.2312  \\
    $C_{9} $ & 100 &  50 & 0.3477  & 1.598  & 0.3945  & 0.2713  & 0.1984  &  0.1984  \\
    $C_{10} $ & 100 &  50 & 0.2422  & 1.4178  & 0.2591  & 0.1735  & 0.133  &  0.133  \\
    \hline
    $C_{11} $ & 100 &  50 &  0.3715 & 1.5703 & 0.382 & 0.2607  & 0.1818  &  0.1818  \\
    $C_{12} $ & 100 &  50 & 0.3114 & 1.3027 & 0.3289 & 0.1818  &  0.1289 &  0.1289  \\
    $C_{13} $ & 100 &  50 & 0.3859  & 1.3201  & 0.4287  &  0.30696 & 0.215  &  0.215  \\
    $C_{14} $ & 100 &  50 &  0.3447 & 1.5994  & 0.3923  & 0.2725  & 0.2182  &  0.2182  \\
    $C_{15} $ & 100 &  50 & 0.2579 & 1.4405  & 0.2839  & 0.2072  & 0.1635  &    0.1635\\
    \hline
    $C_{16} $ & 100 &  50 &  0.7153 & 10.628 & 0.8203 & 0.8202  & 0.8202  &  0.8201  \\
    $C_{17} $ & 100 &  50 & 0.7309 & 10.641 & 0.8124 & 0.8123  &  0.8123 &  0.8122  \\
    $C_{18} $ & 100 &  50 & 0.8497  & 10.788  & 0.9727  &  0.9727 & 0.9727  &  0.9724  \\
    $C_{19} $ & 100 &  50 & 0.9758 & 10.907 & 1.1094  &  1.1093 & 1.1093  & 1.1090    \\
    $C_{20} $ & 100 &  50 & 0.7178 & 10.647  & 0.8106  & 0.8105  &  0.8105 &    0.8103
\end{tabular}
    \caption{Entropy gaps: $s\sim n/2$}\label{tab:table1}
  \end{center}
\end{table}

We implemented our algorithm in \texttt{Matlab},
and we carried out some  experiments designed to demonstrate the potential of our approach, compared to  
the masked linx bound using the $\frac{1}{2}$-mask.
Our test instances are described in Appendix 2.
In our experiments presented in this section, we took $s\sim  n/2$.\footnote{Results of further experiments, with $s\sim n/4$ and $s\sim 3n/4$, are presented in Appendix 3.} %$^{\scriptscriptstyle \ref{note1}}$
Our results appear in Table \ref{tab:table1}, where we tabulated (difference)
gaps to a heuristically-generated lower bound on $z(C,s)$ (see \cite[Section 4]{KLQ}).
``$C$" indicates the covariance matrix that we used.
``$\mbox{linx}(\frac{1}{2})$'' is the gap value of the masked linx bound using the $\frac{1}{2}$-mask.
``$\lambda(\frac{1}{2})$'' is the spectral-bound gap using the $\frac{1}{2}$-mask.
Each remaining  column starts from the solution of the previous column.
``$\lambda(\Pi)$''  is the spectral-bound gap after performing a first-phase (permutation-based local search).
``$\lambda(\Sigma)$'' is the spectral-bound gap after performing a second phase (signature-based local search).
``$\mbox{linx}(\Sigma)$'' is the linx-bound gap after performing a second phase (signature-based local search).
``$\mbox{DP}(\Sigma)$'' is the DP-bound gap after performing a second phase (signature-based local search).
In our experiments, we can see clear and consistent
improvements from $\lambda(\frac{1}{2})$ to $\lambda(\Pi)$
to $\lambda(\Sigma)$ to $\mbox{linx}(\Sigma)$. 
Furthermore, we can see that  $\mbox{linx}(\Sigma)$ improves on $\mbox{linx}(\frac{1}{2})$, except for $C_{16}$ to $C_{20}$. 
Interestingly, for those we can see a small improvement in $\mbox{DP}(\Sigma)$ compared to $\mbox{linx}(\Sigma)$ (in  contrast to the other test problems). 
Overall, our method gives a good way to get a better mask for
the linx bound than the $\frac{1}{2}$-mask.
Considering the final column, $\mbox{DP}(\Sigma)$, 
%(the entries are 
%less than the preceding column by positive amounts that are less than $10^{-5}$),
we can see that the masked linx bound using the final mask $M$, consistently gives almost
exactly $z(C\circ M,s)$. That is, the dynamic program establishes that
the masked linx bound with the final mask $M$ nearly gives the optimal value of $z(C\circ M,s)$.

The plot of Figure \ref{fig:gap_0.5n} shows the average portion of the gap between $\lambda(\frac{1}{2})$ and $DP(\Sigma)$ that is bridged by each
$\Delta\in\{ \lambda(\frac{1}{2}), \lambda(\Pi), \lambda(\Sigma),  \mbox{linx}(\Sigma), \mbox{DP}(\Sigma)\}$. That is, we plot averages of $\frac{\lambda(\frac{1}{2}) - \Delta}{\lambda(\frac{1}{2})-\mbox{\scriptsize DP}(\Sigma)}$; averaged over each of the five
types of instances. In this way, we can understand the improvements obtained by
successively adding more bounding effort.

% $\lambda(\Pi)$, $\lambda(\Sigma)$,  $\mbox{linx}(\Sigma)$ and $DP(\Sigma)$ respectively. Each fraction is calculated by looking at the difference between each method and {\color{red} the one used directly before it}, divided by the total gap between $\lambda(\frac{1}{2})$ and $DP(\Sigma)$. 
%
% {\color{red}
% Is that right? Isn't it $\frac{\lambda(\frac{1}{2}) - \Delta}{\lambda(\frac{1}{2})-DP(\Sigma)}$,
% where $\Delta\in\{ \lambda(\frac{1}{2}), \lambda(\Pi), \lambda(\Sigma),  \mbox{linx}(\Sigma), DP(\Sigma)\}$?
% }

\begin{figure}
    \centering
    \includegraphics[width =280pt]{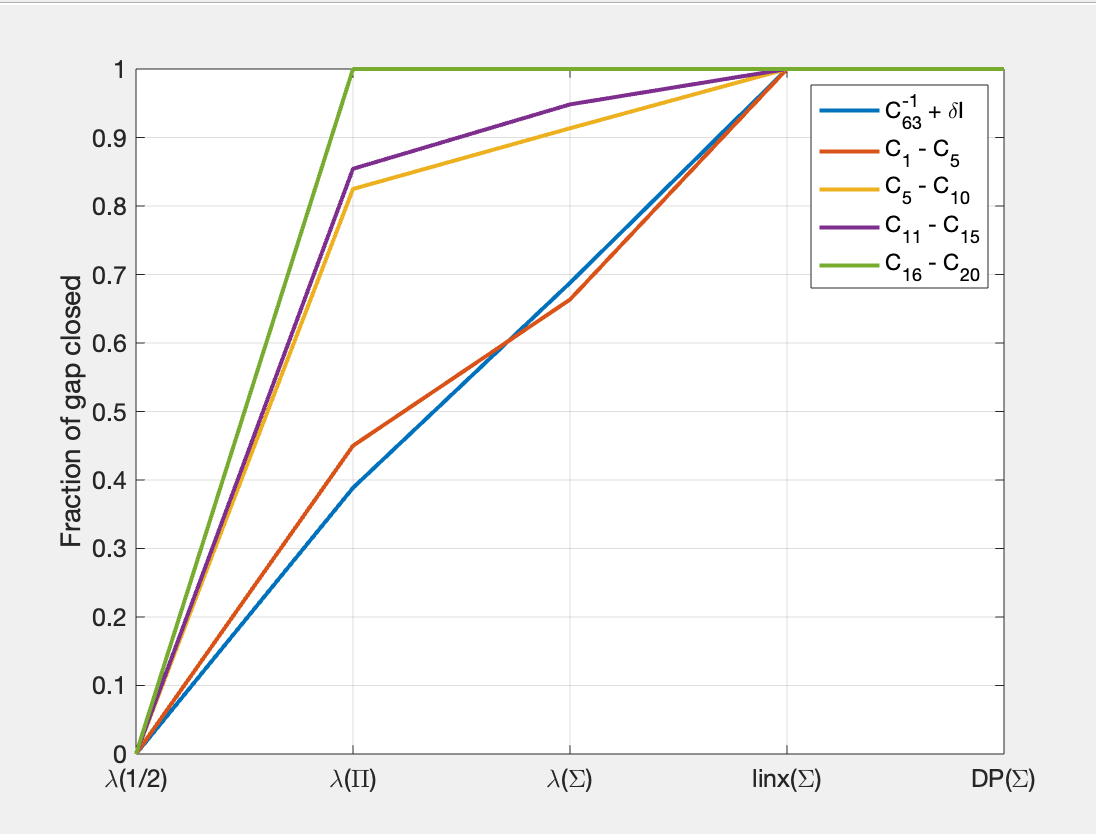}
    \caption{Average fraction of gap closed by each masking phase}
    \label{fig:gap_0.5n}
\end{figure}

\FloatBarrier

\section*{Appendix 1: Deferred proofs} 

\subsection*{Proof of Lemma \ref{lem:det}}

% Recall that for $1 \leq p < q < n$, $M(n,p,a,q,b)$ be a $1/2$-mask of order $n$, except the entries in positions $(p,p+1)$ and $(p+1,p)$ are $a$ and the entries positions $(q,q+1)$ and $(q+1,q)$ are $b$.
% \medskip

% \noindent {\bf Lemma \ref{lem:det}.}
% For $1\leq p<q < n$, and $a,b\in \mathbb{R}$, 
% \begin{align*}
% \det M(n,p,a,q,b) = 
% &\dfrac{1}{2^{n}} (n-q+1)\Big( {(p+1)(q-p+1) - 4a^2p(q-p)}\Big) \\
% &\quad -\dfrac{1}{2^{n}} (n-q)4b^2\left((p+1)  (q-p)- 4a^2p(q-p-1)\right) .
% \end{align*}

\noindent Our proof is by induction on $n$, and our base cases are $n=3$ and $n=4$.
For $n=3$, we have
\begin{align*}
\det M(3,p,a,q,b) &= \det M(3,1,a,2,b) = (1-a^2) - b^2\\
&= \dfrac{1}{2^{3}} (3-2+1)\Big( {(1+1)(2-1+1) - 4a^2(2-1)}\Big) \\
&\quad -\dfrac{1}{2^{3}} (3-2)4b^2\left((1+1)  (2-1)- 4a^2(2-1-1)\right).
\end{align*}
\noindent
For $n=4$, we have three combinations of $p$ and $q$ to consider.\\

\noindent
Case 1: $p = 1$, $q = 2$.
\begin{align*}
\det M(4,1,a,2,b) &=  \frac{3}{4}(1-a^2) - b^2 \\
&= \dfrac{1}{2^{4}} (4-2+1)\Big( {(1+1)(2-1+1) - 4a^2(2-1)}\Big) \\
&\quad -\dfrac{1}{2^{4}} (4-2)4b^2\left((1+1)  (2-1)- 4a^2(2-1-1)\right).
\end{align*}

\noindent
Case 2: $p = 1$, $q = 3$.
\begin{align*}
\det M(4,1,a,3,b) &=  (1-a^2)(1-b^2) - \frac{1}{4} \\
&= \dfrac{1}{2^{4}} (4-3+1)\Big( {(1+1)(3-1+1) - 4a^2(3-1)}\Big) \\
&\quad -\dfrac{1}{2^{4}} (4-3)4b^2\left((1+1)(3-1)- 4a^2(3-1-1)\right). 
\end{align*}

\noindent
Case 3: $p = 2$, $q = 3$.
\begin{align*}
\det M(4,2,a,3,b) &=  \frac{3}{4}(1-b^2) - a^2 \\
&= \dfrac{1}{2^{4}} (4-3+1)\Big( {(2+1)(3-2+1) - 4a^2(2)(3-2)}\Big) \\
&\quad -\dfrac{1}{2^{4}} (4-3)4b^2\left((2+1)(3-2)- 4a^2(2)(3-2-1)\right).
\end{align*}

Next, we suppose that $n>4$ and that the result holds for matrices of order less than $n$. 
We apply the recursive determinant formula for tridiagonal matrices (Lemma \ref{lem:tridet}) to $M(n,p,a,q,b)$. Because the recursion has a depth of two, we need to consider six combinations of  $p$ and $q$.

Recall our short notation
$M(n,p,a) := M(n,p,a,q,1/2) = M(n,p_1,1/2,p,a)$.
\begin{align*}
   \det M(n,p,a)& =  \dfrac{1}{2^{n}} (n-p)\Big( {(p+1)(n-p) - 4a^2p(n-p-1)}\Big) \\
     &\quad-  \dfrac{1}{2^{n}} (n-p-1)4(1/2)^2\Big( {(p+1)(n-p) - 4a^2p(n-p-1)}\Big)\\
     & =\dfrac{1}{2^{n}} \Big( {(p+1)(n-p+1) - (n-p)4a^2p}\Big).
\end{align*}

\noindent
Case 1: $p < q < n-2$.
    \begin{align*}
    \det M(n,p,a,q,b) &= \det M(n-1,p,a,q,b) - \frac{1}{2^2}\det M(n-2,p,a,q,b)\\
    & = \dfrac{1}{2^{n-1}} (n-q)\Big( {(p+1)(q-p+1) - 4a^2p(q-p)}\Big) \\
    &\qquad -\dfrac{1}{2^{n-1}} (n-q-1)4b^2\left((p+1)  (q-p)- 4a^2p(q-p-1)\right)\\
     &\quad -  \dfrac{1}{2^{n}} (n-q-1)\Big( {(p+1)(q-p+1) - 4a^2p(q-p)}\Big) \\
    &\qquad +\dfrac{1}{2^{n}} (n-q-2)4b^2\left((p+1)  (q-p)- 4a^2p(q-p-1)\right)\\
    & = \dfrac{1}{2^{n}} (n-q+1)\Big( {(p+1)(q-p+1) - 4a^2p(q-p)}\Big) \\
    &\quad -  \dfrac{1}{2^{n}} (n-q)4b^2\Big( {(p+1)(q-p) - 4a^2p(q-p-1)}\Big).
    \end{align*}

\noindent
Case 2:  $p < n-3$, $q= n-1$.
    
    \begin{align*}
    \det M(n,p,a,n-1,b) 
    & = \det M(n-1,p,a) - b^2\det M(n-2,p,a)\\
    & = \dfrac{1}{2^{n-1}} \Big( {(p+1)(n-p) - (n-p-1)4a^2p}\Big) \\
    & \quad - \dfrac{b^2}{2^{n-2}} \Big( {(p+1)(n-p-1)- (n-p-2)4a^2p}\Big) \\
    & = \dfrac{1}{2^{n}} 2\Big( {(p+1)(n-p) - (n-p-1)4a^2p}\Big) \\
    &\quad -  \dfrac{1}{2^{n}} 4b^2\Big( {(p+1)(n-p-1) - (n-p-2)4a^2p}\Big).
    \end{align*}    
    
\noindent
Case 3: $p = n-3$, $q = n-1$.
    
    \begin{align*}
       & \det M(n,n-3,a,n-1,b) = \det M(n-1,n-3,a) - b^2\det M(n-2,n-3,a)\\
        &\quad = \dfrac{1}{2^{n-1}} \Big( {(n-2)(3) - (2)4a^2(n-3)}\Big) 
        - \dfrac{b^2}{2^{n-2}} \Big( {(n-2)(2)- 4a^2(n-3)}\Big) \\
        & \quad = \dfrac{1}{2^{n}} 2\Big( {(n-2)(3) - (2)4a^2(n-3)}\Big) 
        -  \dfrac{1}{2^{n}} 4b^2\Big( {(n-2)(2) - 4a^2(n-3)}\Big).
    \end{align*}    
    
\noindent
Case 4: $p = n-2$, $q = n-1$.

    \begin{align*}
       & \det M(n,n-2,a,n-1,b)  = \det M(n-1,n-2,a) - b^2\det M(n-2,n-3,1/2)\\
        & \quad = \dfrac{1}{2^{n-1}} \Big( {(n-1)(2) - 4a^2(n-2)}\Big)  - \dfrac{b^2}{2^{n-2}} \Big( {(n-2)(2)- (n-3)}\Big) \\
        &\quad = \dfrac{1}{2^{n}} 2\Big( {(n-1)(2) - 4a^2(n-2)}\Big)  -  \dfrac{1}{2^{n}} 4b^2\Big( {n-1}\Big).
    \end{align*}

\noindent
Case 5: $p < n-3$, $q = n-2$.
    \begin{align*}
        \det M(n,p,a,n-2,b) &= \det M(n-1,p,a,n-2,b) - (1/2)^2\det M(n-2,p,a)\\
        & = \dfrac{1}{2^{n-1}} 3 \Big( {(p+1)(n-p-1) - (n-p-2)4a^2p}\Big) \\
        &\qquad  - \dfrac{b^2}{2^{n-2}} (2)\Big( {(p+1)(n-p-2)- (n-p-3)4a^2p}\Big) \\
        &\quad - \dfrac{1}{2^{n}}\Big( {(p+1)(n-p-1) - (n-p-2)4a^2p}\Big) \\
        & = \dfrac{1}{2^{n}} 3\Big( {(p+1)(n-p-1) - (n-p-2)4a^2p}\Big) \\
        &\quad -  \dfrac{1}{2^{n}} (2)4b^2\Big( {(p+1)(n-p-2) - (n-p-3)4a^2p}\Big).
    \end{align*}   
    
\noindent
Case 6:  $p=n-3$, $q=n-2$.
    
    \begin{align*}
        &\det M(n,n-3,a,n-2,b)\\
        & \quad = \det M(n-1,n-3,a,n-2,b) - (1/2)^2\det M(n-2,n-3,a)\\
        &\quad  = \dfrac{1}{2^{n-1}} (2)\Big( {(n-2)(2) - 4a^2(n-3)}\Big) -\dfrac{1}{2^{n-1}} 4b^2\left(n-2\right) \\
        &\qquad - \dfrac{1}{2^{n}} \Big( {(n-2)(2)- 4a^2(n-3)}\Big) \\
        & \quad = \dfrac{1}{2^{n}} (3)    \Big( {(n-2)(2) - 4a^2(n-3)}\Big) 
         -  \dfrac{1}{2^{n}}(2) 4b^2\Big( {(n-2)}\Big).
    \end{align*}   
  \hfill  \qed

%%%%%%%%%%%%%%%%%%%%%%%%%%%%

\subsection*{Proof of Theorem \ref{thm:bp1}}

\noindent 
First, we note that it is easy to check that  $a^*(n,p)$ is convex and symmetric
about $n/2$ in $p$. So $a^*(n,p)$
is maximized on $[1,n-2]$ at $p=1$.
Therefore, for $a\in (\frac{1}{2},a^*(n,p)]$,
we also have $a\in (\frac{1}{2},a^*(n,1)]$.
% Hence, it is meaningful to consider 
% $b^*(n,1,a,q)$. 
Hence, $b^*(n,1,a,q)$ is well defined (for $a\in (\frac{1}{2},a^*(n,p)]$).

We begin by investigating where
the continuous function $\beta(p):=(b^*(n,p,a,q))^2$ is decreasing, for (continuous) $p \in (1, n-2)$.

\begin{align*}
  \frac{\partial \beta(p)}{\partial p} = &\frac{\left(4 a^2-1\right)(n-q+1)}{4(n-q)}\frac{ \left((2a+1)p+1\right)\left((2a-1)p-1\right)}{\left(\left(4
   a^2-1\right) p (p-q+1)+q\right)^2}
\end{align*}
It is clear that all factors in the numerator and denominator are positive
except for $(2a-1)p-1$, which we need to analyze. We can easily see that this
is negative, precisely when $a<\frac{1}{2}\left(1+\frac{1}{p}\right)$.
So, we only need check that $a^*(n,p)< \frac{1}{2}\left(1+\frac{1}{p}\right)$; that is,
we need to check that 
\[
 \frac{1}{2}\sqrt{\left(1+\frac{1}{p}\right)\left(1+\frac{1}{n-p}\right)} < \frac{1}{2}\left(1+\frac{1}{p}\right).
\]
But this easily reduces to $1+\frac{1}{n-p} < 1+ \frac{1}{p}$,  which is
true precisely when $p<n/2$. 
In particular, $\beta(p)$ is decreasing 
on $(1,n/2)$ and increasing on $(n/2,n-2)$.
So it is quasiconvex, and has a maximum on $[1,q-1]$ at an endpoint.

Next, we will demonstrate  that $\Delta b^*(q) :=  b^*(n,1,a,q) - b^*(n,q-1,a,q) \geq 0$, which will complete our proof. After some algebraic manipulations, we have that $\Delta b^*(q) \geq 0$ simplifies to
% {\color{red}
% \begin{align*}
% \dfrac{2q - 4a^2(q-1)}{
% 2(q-1) - 4a^2(q-2) } & \geq \dfrac{2q - 4a^2(q-1)}{
% q }.
%  \end{align*}
% We aim to show that, 
% \begin{align}
%     \dfrac{q}{
% \Big(2(q-1) - 4a^2(q-2) \Big)} & \geq 1 \label{eq:wts}
% \end{align}
% so first we verify that $ 2q - 4a^2(q-1) \geq 0$. First note that $q = \frac{2a^2}{2a^2 - 1}$ is a root of $ 2q - 4a^2(q-1)$. We also know that for $a < \frac{\sqrt(2)}{2} \implies$ $\frac{2a^2}{2a^2 - 1} < 0$ and for $a^*(n,1) \leq a < \frac{\sqrt(2)}{2}$ we show that $\frac{2a^2}{2a^2 - 1} \geq n$. Since $\frac{2a^2}{2a^2 - 1}$ is decreasing in $a$ for $a > \frac{\sqrt(2)}{2}$, it is sufficient to show that $\frac{2a^{*2}(n,1)}{2a^{*2}(n,1) - 1} \geq n$. This can be easily verified by plugging in the formula in proposition \ref{prop:astar}. % some algebra commented out below
% % \begin{align*}
% %     a^*(n,1) = \frac{1}{2}\sqrt{\dfrac{2n}{n-1}}\\
% %     q = \frac{\dfrac{n}{n-1}}{\dfrac{n}{n-1} - 1}
% %      = n
% % \end{align*}
% This implies that the root is either negative or larger than $n-1$. Now we can fix $q=n-1$ and analyze $ 2q - 4a^2q + 4a^2 $ on the interval for $a \in (1/\sqrt{2},a^*(n,1)]$. Note that for this interval $ 2q - 4a^2q + 4a^2 $ is decreasing in $a$. So it is sufficient to check that $ 2q - 4a^{*2}(n,1)q + 4a^{*2}(n,1) \geq 0 $.
% \begin{align*}
%     2q - 4a^{*2}(n,1)q + 4a^{*2}(n,1) &= 2q - \dfrac{2nq}{n-1} + \dfrac{2n}{n-1}\\
%     & =  2q + \dfrac{ 2n(1-q)}{n-1} = \dfrac{ 2(n - q)}{n-1} > 0
% \end{align*}
% }
\begin{align}\label{dbineq}
\dfrac{(1-2a^2)q+2a^2}
%{2(q-1) - 4a^2(q-2) }  
{2( (1-2a^2)q + (4a^2  - 1))}
\geq
\dfrac{(1-2a^2)q+2a^2}{
q }.
\end{align}
Claim 1:
$(1-2a^2)q+2a^2 \geq 0$. To see this, first note that it is clear when $a\leq \sqrt{2}/2$.
If $a> \sqrt{2}/2$, then the claim is equivalent to $q \leq \frac{a^2}{a^2-1/2}$. The left-hand side of this last expression is trivially increasing in $q$ and the right-hand side is decreasing in $a$, so it suffices
to check that 
\[
n-1 \leq \frac{(a^*(n,1))^2}{(a^*(n,1))^2 -1/2}.
\]
Using Proposition \ref{prop:astar}, we can check that the right-hand side of this last expression is precisely $n$, and so the claim is verified.

%   {\color{red} 
% Finally we verify
% $\frac{q}{
%  2(q-1) - 4a^2(q-2) }  \geq 1 $
% by showing
% \begin{align*}
%  0 \leq 2( q(1-2a^2) + 4a^2  - 1 )\leq q.
% \end{align*}
% The first inequality can be verified using the same steps taken to show that $(1-2a^2)q+2a^2 \geq 0$. Note that the first inequality is true for $a \leq \frac{\sqrt{2}}{2}$. For $a > \frac{\sqrt{2}}{2}$, the statement is equivalent to $q \leq \frac{4a^2-1}{2a^2-1}$. The right-hand side of the last expression is decreasing in $a$, so it suffices to show that 
%  $\frac{4(a^{*}(n,1))^2-1}{2(a^{*}(n,1))^2-1} > n $.
%  \begin{align*}
%     n-1 \leq \dfrac{4(a^{*}(n,1))^2-1}{2(a^{*}(n,1)^2)-1} 
% \end{align*}
% Using Proposition \ref{prop:astar} we can confirm that the right-hand side is precisely $n+1$.
%     }

Claim 2: $  (1-2a^2)q + (4a^2  - 1) >0$.  It is clearly true when $1/2\leq a \leq \sqrt{2}/2$.
For $a > \frac{\sqrt{2}}{2}$, the claim is equivalent to $q \leq \frac{4a^2-1}{2a^2-1}$.
Similarly to Claim 1, it suffices to check that 
\[
n<\frac{4(a^{*}(n,1))^2-1}{2(a^{*}(n,1))^2-1}.
\]
Using Proposition \ref{prop:astar}, we can check that the right-hand side of this last expression is precisely $n+1$, and so the claim is verified.

%details in claim 2, commented at line ~794 (might move  after edits)
% {\color{red}The second inequality can be simplified to,
% \begin{align*}
%     (1 - 4a^2)(q-2) \leq 0
% \end{align*}
% this statement is true for the interval of $a\geq \frac{1}{2}$ and $q \geq 2$. The result follows.
% }

Combining Eq.~\eqref{dbineq}, and the two claims, it remains to show that 
$2( (1-2a^2)q + (4a^2  - 1))\leq q $. But this reduces to showing that 
$(1 - 4a^2)(q-2) \leq 0$, which immediately follows from $a\geq 1/2$ and $q\geq 2$.

\hfill\qed

%  The roots of $\Delta b^*(q)$ are at $q = 2$ and  $q= \frac{2a^2}{2a^2 -1}$.
% % The second root is always negative, because $a<\sqrt{2}/2$.
% {\color{blue} The second root is always less than 2 for $a \in [1/2,1)$}
% {\color{red} That doesn't seem to be true. Take $a$ slightly larger than $\sqrt{2}/2$.}
%  Therefore, $\Delta b^*(q)$ is either nonnegative for all $q$ in $[2,n-1]$
%  or nonpositive for all $q$ in $[2,n-1]$. 
% But it must be nonnegative, because~
% %  Then it is sufficient to show that $b^*(n,1,a,q) - b^*(n,q-1,a,q) \geq 0$ for some point $q > 2 $, let $q= \min\{n\}-1 = 3$.
% \begin{align*}
% \Delta b^*(3) & = \frac{\left(4a^2-3\right) \left(4 a^2-1\right)}{6 \left(a^2-1\right)}
% \end{align*} 
% is positive for all $ a^*(n,p) \geq a > 1/2$. {\color{purple}$a>1/2$. there's an interval of $a>1/2$ where this is negative.}%\sqrt{\frac{3}{4}} > a^*(4,1) \geq

%%%%%%%%%%%%%%%%%%%%%
%%%%%%%%%%%%%%%%%%%%%

\subsection*{Proof of Theorem \ref{thm:end}}

\noindent
Let $\beta(q):=(b^*(n,1,a,q))^2$.
We first demonstrate that $\beta(q)$ is quasiconvex in $q$. This implies that the maximum is attained a boundary point, that is $q=p+1$
or $q=n-1$. 
First, we calculate the unique stationary point
% \begin{align*}
% \bar{q}(a) =  \frac{2 a^2 (n+2)-n-1}{2 (2a^2-1)}. 
% \end{align*}
\begin{align*}
\bar{q}(a) :=  \frac{n}{2} + 1 + \frac{1}{2(2a^2-1)}. 
\end{align*}
Next, we demonstrate that $\bar{q} \notin [\frac{n}{2} + 1, n-1]$. It is easy to see that $\bar{q}$ is decreasing in $a$.
%by inspecting the derivative of $\bar{q}(a)$. 
Then, for $a \in [1/2, \sqrt{2}/2]$, $\bar{q}$ is maximized at $a = 1/2$, and we have $\bar{q}(a) \leq \bar{q}(1/2)= n/2$.
%for $a \in [1/2, 1/\sqrt{2}]$. 
For $a \in (\sqrt{2}/2, a^*(n,1)]$, $\bar{q}$ is minimized at $ a= a^*(n,1)$, and we have $\bar{q}(a) \geq \bar{q}(a^*(n,1)) = n + 1/2$. This implies that for all $a \in [1/2,a^*(n,1)]$, we have $\bar{q} \notin [\frac{n}{2} + 1, n-1]$.

Next, we will demonstrate that 
$\frac{\partial \beta}{\partial q} > 0 $ for $q \geq \frac{n}{2} + 1$.
We consider the case of $q=n-1$, the maximum value  for $q$.
Then we can calculate $\frac{\partial \beta}{\partial q}$ evaluated
at $q=n-1$: 
% \[
% \frac{\left(2 a^2 (n-4)-n+3\right) \left(2 a^2 (n-1)-n\right)}{4 \left(n-2 a^2(n-3)-2\right)^2} > 0.
% \]
\[
\frac{\left((2 a^2 -1)(n-4)-1\right) \left((2 a^2-1) (n-1)-1\right)}
%{4 \left(n-2 a^2(n-3)-2\right)^2} > 0.
{4 \left((2a^2-1)(n-3)-1\right)^2} > 0.
\]
Note that is easy to check that  the numerator is positive for $n \geq 4$, and $a \in [\frac{1}{2},  a^*(n,1)]$. We do this by observing that the two positive roots of the numerator are $\sqrt{\frac{n}{2(n-1)}}$ and $\sqrt{\frac{n-3}{2(n-4)}}$, both of which are greater than or equal to $a^*(n,1) = \sqrt{\frac{n}{2(n-1)}}$. Then we can plug in any relevant $a$ and observe that both the factors in the numerator are negative. This also implies that if $\bar{q} \in [2,\frac{n}{2} + 1]$, then $\beta(q)$ is nondecreasing for the interval $[\bar{q},n-1]$ and either nonincreasing or nondecreasing for $q \in [2,\bar{q}]$. For $\bar{q} \notin [2,n/2+1]$, $\beta(q)$ is strictly increasing.
Because $\beta(q)$ has only one stationary point, at $\bar{q}$,
we can conclude that $\beta(q)$ (and hence $b^*(n,1,a,q)$) is quasiconvex in $q$, and therefore the maximum is either at $q=2$ or $q=n-1$. 

%Next, we wish to demonstrate that $b^*(n,1,a,n-1) \geq  b^*(n,1,a,2)$, which is 
%equivalent to 
%As a first step we simplify the statement we want to prove 
It remains to demonstrate that
$b^*(n,1,n-1,a) \geq  b^*(n,1,2,a)$, which reduces to demonstrating that
%into a statement that's easier to evaluate:
\begin{align*}
   \frac{a^2 (n-3) \left(2 a^2 (n-1)-n\right)}{2 (n-2) \left(2 a^2 (n-3)-n+2\right)} & \geq 0. %\tag{1 }\label{eq. \ref{eq:recur}}
\end{align*}
The positive root of the numerator is $\sqrt{\frac{n}{2(n-1)}}$, which is greater than $a^*(n,1)$, and then it is easy to check then that the numerator is 
nonpositive for all  $a \in [1/2, a^*(n,1)]$. Similarly, the positive 
root of the denominator is  $\sqrt{\frac{n-2}{2(n-3)}}$, 
 which is greater than $a^*(n,1)$, and then it is easy to check then that the numerator is 
negative for all  $a \in [1/2, a^*(n,1)]$. The result follows.
% Since the remaining factors are positive, we concern ourselves with ensuring $(2 a^2 (n-1)-n)$ and $(2 a^2 (n-3)-n+2)$ are both nonpositive for $a \in [1/2, a^*(n,1)]$. We do this by observing that the two positive roots of the relevant factors are $\sqrt{\frac{n}{2(n-1)}}$ and $\sqrt{\frac{n-2}{2(n-3)}}$, both of which are greater than $a^*(n,1)$. Then it is easy to check that both factors are nonpositive for any relevant $a \in [1/2,a^*(n,1)]$.
%  It is simple to verify that 
% $\left(2 a^2 (n-1)-n\right) \leq 0$ and $\left(2 a^2 (n-3)-n+2\right) < 0$ for $a =a^*(n,1)$.
% Observing that the remaining factors are positive, the result follows.
% \begin{align*}
% 2 (a^{*}(n,1))^2 (n-3)-n+2 & = 2n\dfrac{n}{2(n-1)} - 2\dfrac{n}{2(n-1)} - n\\
% & = \dfrac{n(n-1)}{n-1} - n=0
% \end{align*}
% \begin{align*}
% 2 (a^{*}(n,1))^2 (n-3)-n+2 & = 2 - n - \dfrac{6n}{2(n-1)}  +  \dfrac{2n}{2(n-1)}  \\
%         & =  -\dfrac{2}{n-1} < 0
% \end{align*}
% Then equation \ref{eq:1} is true and so is $b^*(n,1,n-1,a) \geq  b^*(n,1,2,a) $.

\hfill\qed

\section*{Appendix 2: Test instances}

Referring to the first column of 
Table \ref{tab:table1}, we explain how we built our 20 test instances:
\begin{itemize} 
\item $C_{63}$ is a well-known benchmark
covariance matrix generated from 63 chemical data sensors (see \cite{KLQ,Kurt_linx}), for example).
To get some idea of how our results can change when uncorrelated noise is added, we
also experimented with adding different positive multiples of the identity matrix.
We can see that the behavior is similar, for different multiples, while the gaps are smaller.
%%%%%%%%%%%%%%%%%%%%%%%%%%%%%%%%%%%%%%%%%%
\item $C_1$ through $C_5$ were generated using the \texttt{Matlab} \texttt{sprandsym} function. This function can take eigenvalues as input; 
for each randomly generated positive-definite matrix we set $\lambda_i := 4^{\frac{n+1-2i}{n-i}}$,
%with $\rho:=4$, 
for each $i \in N$. This gives a nice concave decreasing sequence of eigenvalues that is 
preserved under matrix inversion (see \cite{Mixing}).
%%%%%%%%%%%%%%%%%%%%%%%%%%%%%%%%%%%%%%%%%%
\item
Each of $C_6$ through $C_{10}$ were generated as follows.
First, we randomly generated a diagonally dominant tridiagonal matrix $\bar{C}$ with:
\vspace{-10pt}
\begin{align*}
& \bar{C}_{i,i} \sim U[2,5],\quad \mbox{ for } i=1,\ldots,n;\\
& \bar{C}_{i,i+1} = \bar{C}_{i+1,i} \sim U[-1,1],\quad \mbox{ for } i=1,\ldots,n-1,
\end{align*}
independently generated. 
Next, we made $m=100,000$ independent samples from the multivariate normal distribution $N(\mathbf{0},\bar{C})$.
From these $m$ samples, we calculated the sample covariance matrix $C$. 
In this way, we get $C$ as a dense noisy version of the tridiagonal $\bar{C}$. 
%%%%%%%%%%%%%%%%%%%%%%%%%%%%%%%%%%%%%%%%%%
\item $C_{11}$ through $C_{15}$ are generated in a similar way to $C_6,\ldots,C_{10}$ but with a different $\bar{C}$. In this case tridiagonal $\bar{C}$ is made up $n/2$ blocks of 
$J_2=
(
{\tiny
  \begin{array}{cc}
1&1\\
1&1\\
  \end{array}
}
)
$ 
plus independent random samples from $U[0,0.05]$ added to each diagonal element.
Here we have samples from $n/2$ pairs of bivariate normals, with no correlation 
between pairs and high correlation within a pair.
Again, we made $m=100,000$ independent samples from the multivariate normal distribution $N(\mathbf{0},\bar{C})$, and
from these $m$ samples, we calculated the sample covariance matrix $C$. 
%%%%%%%%%%%%%%%%%%%%%%%%%%%%%%%%%%%%%%%%%%
\item For $C_{16}$ through $C_{20}$, we randomly and independently generated
\begin{align*}
& R_i \sim N(0,0.2),\quad \mbox{ for } i=1,\ldots,n;\\
& E_i \sim N(0,1),\quad \mbox{ for } i=0,\ldots,n+1.
\end{align*}
\vspace{-20pt}

\noindent Then, with $d=0.2$,  let 
\[
Y_i := R_i +  E_i+ d\times (E_{i-1} +  E_{i+1}),\quad\mbox{  for } i=1,\ldots,n.
\]
In this way, we have significant correlation between ``neighbors''.  
We made $m=100,000$ independent samples from this distribution, we calculated the sample covariance matrix, and we used this as $C$. 

\end{itemize}

In Table \ref{tab:c_char}, we summarize some statistics for our test matrices.
The ratios $\lambda_1/\lambda_n$ clearly indicate that all of our matrices are
not nearly diagonal. The column `abs sum' tabulates $\Sigma_{i,j} |C[i,j]|$. 
The column `abs sum tridiag' tabulates 
 $\Sigma_{i-j\leq 1} |C[i,j]|$. Hence, the ratios between these two, the tri-ratio' values,
 measures how far from tridiagonal these matrices are. 
 We can see that the the first five have only a very modest degree of tridiagonality. 
 $C_1$--$C_5$ are not tridiagonal at all. 
 Finally, $C_6$--$C_{20}$  are decidedly not
 tridiagonal, but not extremely far from being tridiagonal, as desired. 
 
\begin{table}[!ht]
    \centering
    \begin{tabular}{c|c|c|c|c|c|c|c}
             & &  &  &  &  & abs sum & tri- \\
        $C$ & $n$ & $\lambda_n(C)$ & $\lambda_1(C)$ & $\lambda_1/\lambda_n$ & abs sum &  tridiag & ratio\\
        \hline
        $C_{63}^{-1}$ & 63& 0.65 & 31.50 & 48.42 & 1780.66 & 920.04 & 1.94 \\
        $C_{63}^{-1} + 0.125I$& 63& 0.78 & 31.63 & 40.78 & 1788.54& 927.92& 1.93 \\
        $C_{63}^{-1}1 + 0.25I$ & 63& 0.90 & 31.75 & 35.26 & 1796.41& 935.79 & 1.92 \\
        $C_{63}^{-1} + 0.375I$ & 63 & 1.03 & 31.98 & 31.18 & 1804.29 & 943.67 & 1.91 \\
        $C_{63}^{-1} + 0.5I$ & 63& 1.15 & 32.00 & 27.81 & 1812.16& 951.54 & 1.90 \\
        \hline
        $C_1$ & 100 & 0.25 & 4.00 & 16.00 & 582.96 & 146.62 & 3.98\\
        $C_2$ & 100 & 0.25 & 4.00 & 16.00 & 620.53 & 143.78 & 4.32\\
        $C_3$ & 100 & 0.25 & 4.00 & 16.00 & 543.63 & 146.14 & 3.72 \\
        $C_4$ & 100 & 0.25 & 4.00 & 16.00 & 580.94 & 144.61 & 4.02 \\
        $C_5$ & 100 & 0.25 & 4.00 & 16.00 & 574.24 & 143.15 & 4.01 \\
        \hline
        $C_6$ & 100 & 1.16 & 5.74 & 4.97 & 546 & 459.67& 1.19\\
        $C_7$ & 100 & 1.41 &5.74 & 4.07 & 520.34 & 435.59 & 1.19 \\
        $C_8$ & 100 & 1.38 & 5.91 & 4.27 & 544.32 & 459.57 & 1.18\\
        $C_9$ & 100 & 1.17 & 5.78 & 4.94 & 537.27 & 451.79 & 1.19\\
        $C_{10}$ & 100 & 1.58 & 5.47 & 3.45 & 534.84 & 451.56 & 1.18\\
        \hline
        $C_{11}$ & 100 & 1.11 & 5.88 & 5.31 & 524.93 & 442.32 & 1.19 \\
        $C_{12}$ & 100 & 1.42 & 5.83 & 4.09 & 518.17 & 435.43 & 1.19\\
        $C_{13}$ & 100 & 1.38 & 5.90 & 4.29 & 546.19 & 459.26 & 1.19\\
        $C_{14}$ & 100 & 1.16 & 5.74 & 4.93 & 535.49 & 452.04 & 1.18 \\
        $C_{15}$ & 100 & 1.59 & 5.45 & 3.43 & 536.84 & 450.93 & 1.19\\
        \hline 
        $C_{16}$ & 100 & 0.49 & 2.51 & 5.12 & 274.69 & 212.78 & 1.29\\
        $C_{17}$ & 100 & 0.49 & 2.53 & 5.14 & 276.40 & 213.01 & 1.30\\
        $C_{18}$ & 100 & 0.49 & 2.52 & 5.13 & 275.30 & 213.04 & 1.29\\
        $C_{19}$ & 100 & 0.49 & 2.52 & 5.11 & 275.19 & 212.97 & 1.29\\
        $C_{20}$ & 100 & 0.49 & 2.51 & 5.09 & 275.68 & 212.94 & 1.29\\[5pt]
    \end{tabular}
    \caption{Test-matrix statistics} \label{tab:c_char}
\end{table}

\FloatBarrier

\section*{Appendix 3: Further experiments}

Figure \ref{fig:morephases} is of the same type as Figure \ref{fig:gap_0.5n},
but now with $s\sim n/4$ and $s\sim 3n/4$. 
A more detailed view of these experiments, following what we presented in Table \ref{tab:table1}, is in Tables \ref{tab:table4} and \ref{tab:table5}. 
While the general trends seen for $s\sim n/2$ persist,
we do see now several instances where $\mbox{DP}(\Sigma)$
improves upon $\mbox{linx}(\Sigma)$.

\begin{figure}[ht]
     \centering
     \begin{subfigure}[b]{0.49\textwidth}
         \centering
         \includegraphics[width=\textwidth]{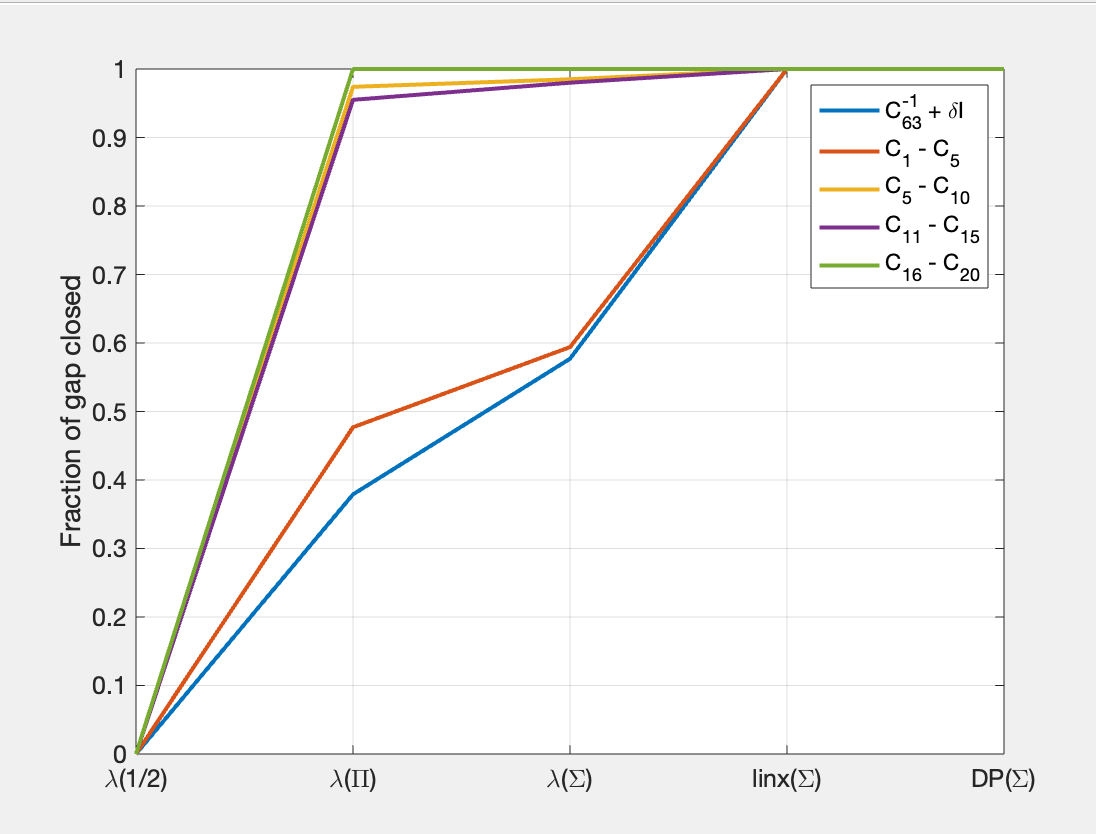}
         \caption{}
         \label{fig:gap_0.25n}
     \end{subfigure}
     \hfill
     \begin{subfigure}[b]{0.49\textwidth}
         \centering
         \includegraphics[width=\textwidth]{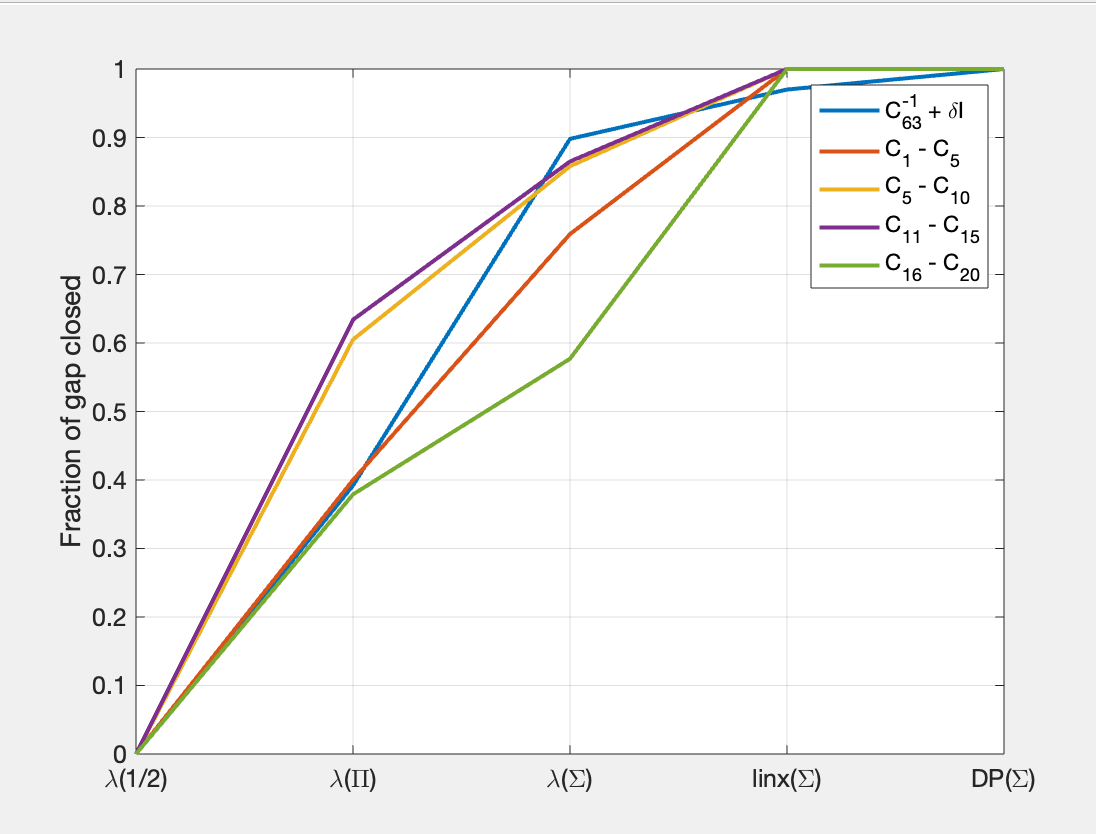}
         \caption{}
         \label{fig:gap_0.75n}
     \end{subfigure}
     \caption{Average fraction of gap closed by each masking phase}\label{fig:morephases}
     \hfill
\end{figure}

\begin{table}[ht]
  \begin{center}
\begin{tabular}{c|c|c|c|c|c|c|c|c}
  $C$ & $n$ & $s$ & $\mbox{linx}(\frac{1}{2})$ & $\lambda(\frac{1}{2})$ & $\lambda(\Pi)$  & $\lambda(\Sigma)$ & $\mbox{linx}(\Sigma)$  & $\mbox{DP}(\Sigma)$      \\
  \hline
   $C_{63}^{-1}$          & 63 &  15 &  0.9839 & 0.9880  & 0.8745  & 0.8150  &  0.6880 & 0.6880 \\
   $C_{63}^{-1} + 0.125I$ & 63 &  15 &  0.9696 & 0.9739  &  0.8616 & 0.8027 & 0.6772 & 0.6772    \\
   $C_{63}^{-1} + 0.25I$  & 63 &  15 &  0.9558 & 0.9601  & 0.8490  & 0.7908  & 0.6667 &  0.2334  \\
   $C_{63}^{-1} + 0.375I$ & 63 &  15 &  0.9423 &  0.9467 & 0.8368  & 0.7791   & 0.6565 & 0.2262 \\
   $C_{63}^{-1} + 0.5I$   & 63 &  15 &  0.9292 & 0.9336  & 0.8248 & 0.7678 & 0.6466 & 0.2192 \\
   \hline
    $C_{1} $ & 100 &  25 & 1.0899 & 1.2264 & 0.9922 & 0.9616  & 0.8183  & 0.8183  \\
    $C_{2} $ & 100 &  25 &  1.0936 & 1.1161 & 00.9034  & 0.7943  & 0.5885  & 0.5885  \\
    $C_{3} $ & 100 &  25 & 0.8229 & 0.8743 & 0.7404  & 0.7144  & 0.6201  & 0.6201  \\
    $C_{4} $ & 100 &  25 & 0.9130 & 0.9969 & 0.7969  & 0.7766  & 0.5616  & 0.5616  \\
    $C_{5} $ & 100 &  25 & 0.8480 & 0.8597 & 0.7065 & 0.6521 & 0.4981 & 0.4981  \\
    \hline
    $C_{6} $ & 100 &  25 &  0.0592 & 0.9574 & 0.0775  & 0.0763  & 0.0656  & 0.0656   \\
    $C_{7} $ & 100 &  25 &  0.0277  & 0.6608  & 0.0193  & 0.0147  & 0.0026  &  0.0026  \\
    $C_{8} $ & 100 &  25 &   0.1028  & 0.9697 & 0.1175  & 0.1168  & 0.0752  & 0.0752  \\
    $C_{9} $ & 100 &  25 & 0.0424 & 0.7216 & 0.0509  & 0.0426  & 0.0426 & 0.0426 \\
    $C_{10} $ & 100 & 25 & 0.0447 & 0.6762 & 0.0435 & 0.0236 & 0.0236 &  0.0236 \\
    \hline
    $C_{11} $ & 100 &  25 & 0.0574 & 0.6121 & 0.0505 & 0.0092  & 0.0057  &  0.0057  \\
    $C_{12} $ & 100 &  25 & 0.0298 & 0.6763 & 0.0195 & 0.0137  &  0.0026 &  0.0023  \\
    $C_{13} $ & 100 &  25 & 0.0853  & 0.9386  & 0.0955  &  0.0930 & 0.0514  &  0.0514  \\
    $C_{14} $ & 100 &  25 & 0.0443 & 0.7356  & 0.0463  & 0.0375  & 0.0146  &  0.0139  \\
    $C_{15} $ & 100 &  25 & 0.0534 & 0.6781  & 0.0541  & 0.0341  & 0.0341  &    0.0341\\
    \hline
    $C_{16} $ & 100 &  25 &  0.3970 & 7.2480 & 0.0462 & 0.0461  & 0.0461  &  0.0461  \\
    $C_{17} $ & 100 &  25 & 0.3929 & 7.2532 & 0.0499 & 0.0498  &  0.0498 &  0.0498  \\
    $C_{18} $ & 100 &  25 & 0.3823  & 7.2506  & 0.0590  &  0.0590 & 0.0588  &  0.0588  \\
    $C_{19} $ & 100 &  25 & 0.3769 & 7.2391 & 0.0497  &  0.0497 & 0.0497  & 0.0496    \\
    $C_{20} $ & 100 &  25 & 0.3949 & 7.2591  & 0.0439  & 0.0438  &  0.0438 &    0.0438
\end{tabular}
    \caption{Entropy gaps: $s\sim n/4$}\label{tab:table4}
  \end{center}
\end{table}

\begin{table}[ht]
  \begin{center}
\begin{tabular}{c|c|c|c|c|c|c|c|c}
  $C$ & $n$ & $s$ & $\mbox{linx}(\frac{1}{2})$ & $\lambda(\frac{1}{2})$ & $\lambda(\Pi)$  & $\lambda(\Sigma)$ & $\mbox{linx}(\Sigma)$  & $\mbox{DP}(\Sigma)$      \\
  \hline
   $C_{63}^{-1}$          & 63 &  47 &  5.1504 & 5.1546  & 4.7257  & 4.1712  &  4.0911 & 4.0582 \\
   $C_{63}^{-1} + 0.125I$ & 63 &  47 &  5.0078 & 5.0120  &  4.5895 & 4.0436 & 3.9649 & 3.9323    \\
   $C_{63}^{-1} + 0.25I$  & 63 &  47 &  4.8736 & 4.8777  & 4.4614  & 3.9239  & 3.8467 &  3.8142  \\
   $C_{63}^{-1} + 0.375I$ & 63 &  47 &  4.7471 &  4.7512 & 4.3409  & 3.8115   & 3.7358 & 3.7035 \\
   $C_{63}^{-1} + 0.5I$   & 63 &  47 & 4.6268  & 4.6308 & 4.2265  & 3.7050 & 3.6307 & 3.6040 \\
   \hline
    $C_{1} $ & 100 &  75 &  14.1800 & 14.2440 & 12.9740  & 11.7390  & 11.2780  & 11.2780  \\
    $C_{2} $ & 100 &  75 &  12.8690 & 12.9300 & 11.8830  & 11.0540  & 10.4730  & 10.4730   \\
    $C_{3} $ & 100 &  75 & 12.0690 & 12.1220 & 10.8660 & 9.6651 & 8.9006 & 8.9006 \\
    $C_{4} $ & 100 &  75 & 13.5630 & 13.5810 & 12.2950 & 11.1710 & 10.2000 & 10.2000   \\
    $C_{5} $ & 100 &  75 & 12.9400 & 13.0260 & 11.7660 & 10.6340 & 9.6674  & 9.6674  \\
    \hline
    $C_{6} $ & 100 &  75 &  1.1464 & 2.2101 & 1.2271  & 0.7309  & 0.6028  & 0.6028  \\
    $C_{7} $ & 100 &  75 &  0.8360  & 1.6736  & 0.8169  & 0.4710  & 0.3384  &  0.3384  \\
    $C_{8} $ & 100 &  75 &   0.9222  & 1.9064 & 0.9357  & 0.5663  & 0.4240  & 0.4240  \\
    $C_{9} $ & 100 &  75 & 1.1488 & 2.1226 & 1.1956  & 0.8275  & 0.4835 & 0.4835 \\
    $C_{10} $ & 100 & 75 & 1.0013 & 1.8240 & 1.0901 & 0.7925 & 0.4964 &  0.4964 \\
    \hline
    $C_{11} $ & 100 &  75 & 0.9316 & 2.3084 & 0.9966 & 0.6492  & 0.4991  &  0.0584  \\
    $C_{12} $ & 100 &  75 & 0.8318 & 1.6619 & 0.8143 & 0.4657 &  0.3406 &  0.3406  \\
    $C_{13} $ & 100 &  75 & 0.9199  & 1.9245  & 0.9319  &  0.5658 & 0.4228  &  0.4228  \\
    $C_{14} $ & 100 &  75 & 1.1488 & 2.1226  & 1.2121  & 0.8175  & 0.5496  &  0.5496  \\
    $C_{15} $ & 100 &  75 & 0.9909 & 1.8258  & 1.0688  & 0.7949  & 0.4852  &    0.4852\\
    \hline
    $C_{16} $ & 100 &  75 &  6.8324 & 14.6930 & 8.4079 & 8.4078  & 8.4072  &  8.4072  \\
    $C_{17} $ & 100 &  75 & 6.8630 & 14.7360 & 8.4535 & 8.4533  &  8.4526 &  8.4526 \\
    $C_{18} $ & 100 &  75 & 6.8601  & 14.7380  & 8.4570  &  8.4570 & 8.4567  &  8.4562  \\
    $C_{19} $ & 100 &  75 & 6.8393 & 14.7240 & 8.4577  &  8.4573 & 8.4567  & 8.4567    \\
    $C_{20} $ & 100 &  75 & 6.8605 & 14.7390  & 8.4501 & 8.4500  &  8.4493 &   8.4493
\end{tabular}
    \caption{Entropy gaps: $s\sim 3n/4$}\label{tab:table5}
  \end{center}
\end{table}

\FloatBarrier

\bibliographystyle{alpha}

\bibliography{AL}

\end{document}